\documentclass[a4paper,10pt]{article}
\usepackage{amsmath}
\usepackage{amsthm}
\usepackage{amsfonts}  
\usepackage{amssymb}
\usepackage{amsmath}
\usepackage{xcolor}
\usepackage{mathrsfs}  
\usepackage{hyperref}
\usepackage{tikz}
\usepackage{enumerate}
\usepackage{verbatim}
\usepackage{vmargin}
\usetikzlibrary{matrix,arrows,shadings}

\title{Stringy invariants for horospherical varieties \\ of complexity one}
\date{\today}

\author{Kevin Langlois\footnote{Heinrich Heine Universit\"at 40225 D\"usseldorf, Germany.}, Cl\'elia Pech\footnote{School of Mathematics, Statistics and Actuarial Science, University of Kent, Canterbury CT2 7FS, United Kingdom.}, Michel Raibaut\footnote{Univ Grenoble Alpes, Univ Savoie Mont Blanc, CNRS, LAMA, 73000 Chamb\'ery., France.}}

\theoremstyle{plain}
\newtheorem{theorem}{Theorem}[section]
\newtheorem{lemma}[theorem]{Lemma}
\newtheorem{proposition}[theorem]{Proposition}
\newtheorem{corollary}[theorem]{Corollary}
\newtheorem*{theorem*}{Theorem}
\newtheorem{propdef}[theorem]{Proposition-Definition}
 
\theoremstyle{definition}
\newtheorem{definition}[theorem]{Definition}

\newtheorem{example}[theorem]{Example}
\newtheorem{notation}[theorem]{Notation}

\theoremstyle{remark}
\newtheorem{remark}[theorem]{Remark}

\def\ZZ{{\mathbb Z}}
\def\NN{{\mathbb N}}
\def\QQ{{\mathbb Q}}
\def\CC{{\mathbb C}}

\def\Var{\textbf{Var}_{\CC}}
\def\Hom{\operatorname{Hom}}
\def\Spec{\operatorname{Spec}}
\def\L{\mathscr{L}}

\def\tronc{\pi}
\def\KO{K_0(\Var)}
\def\LL{\mathbb{L}}
\def\AA{\mathbb{A}}

\def\Mc{\hat{\mathscr{M}}_\CC}

\def\mm{\mu}
\def\ord{\operatorname{ord}}
\def\OO{\mathscr{O}}
\def\I{\mathscr{I}}

\def\E{\mathcal{E}_{st}}
\def\EE{E_{st}}
\def\XX{\mathbb{X}}
\def\D{\mathfrak{D}}
\def\F{\mathcal{F}}

\def\Am{H^0(C_0,\OO_{C_0}(\lfloor \D(m) \rfloor))}
\def\DF{\mathscr{E}}
\def\O{\mathcal{O}}
\def\K{\mathcal{K}}

\def\alphaG{\alpha^*_\Gamma}
\def\alphaM{\alpha^*_M}
\def\p{\mathfrak{p}}
\def\Sc{\mathfrak{S}}
\def\N{\mathscr{N}}
\def\hor{\mathcal{X}_{\mathrm{hor}}}
\def\ver{\mathcal{X}_{\mathrm{ver}}}
\def\verY{\mathcal{X}_{\mathrm{ver},\alpha}}
\def\val{\underline{\operatorname{val}}}
\def\C{\mathcal{C}}
\def\vinf{\varphi_\infty}
\def\bvinf{\bar{\varphi}_\infty}
\renewcommand{\epsilon}{\varepsilon}
\def\Vert{\operatorname{Vert}}
\def\Ray{\operatorname{Ray}}
\def\Div{\operatorname{Div}}
\def\div{\operatorname{div}}
\def\PL{\mathrm{PL}}
\def\PP{\mathbb{P}}
\def\DFd{\DF_{dis}}
\def\Cone{\operatorname{Cone}}
\def\Conv{\operatorname{Conv}}
\def\Sp{\operatorname{Sp}}

\def\ra{\rightarrow}
\def\tor{Q}
\def\qfun{R}
\def\dw{\mathbb{X}_{+}(Q)}

\def\prop{\mathrm{Prop}}

\begin{document}

\maketitle

\begin{abstract}
 We determine the stringy motivic volume of log terminal horospherical varieties of complexity one and obtain a smoothness criterion using a comparison of stringy and usual Euler characteristics.
\end{abstract}

\setcounter{tocdepth}{1}

\section{Introduction}

Motivic integration is a natural geometric counterpart to $p$-adic integration. It is a powerful tool for creating 
new invariants for algebraic varieties, such as stringy invariants, see~\cite{Batyrev1998:StringyHodge}. In \cite{Batyrev-Moreau}, 
stringy invariants were instrumental in simplifying a smoothness criterion for embeddings of horospherical homogeneous spaces; 
however for general spherical varieties the formulation of this criterion remains a conjecture, see \cite[Conjecture 6.7]{Batyrev-Moreau}.

In this paper, we study stringy invariants of $\QQ$-Gorenstein horospherical varieties of complexity one with log terminal singularities. 
One of our results, Theorem \ref{t:stringy-volume}, provides an explicit formula for 
stringy invariants in terms of their combinatorial description. 
We also obtain a rational form for the stringy motivic volume and, using convex geometry, an explicit finite set of candidate poles, 
see Theorem~\ref{t:candidate-poles}. Finally, under additional assumptions, we rephrase in Theorem \ref{t:smoothness} 
a smoothness criterion for these varieties in terms of the stringy Euler characteristic. 

Throughout, we work over complex numbers. 
We denote by $G$ a connected simply connected reductive linear algebraic group and by $B$ a Borel subgroup. By `simply connected' we mean that $G=G^{ss} \times F$, where $G^{ss}$ is a simply connected semisimple group in the usual sense and $F$ is an algebraic torus. A $G$-homogeneous space $G/H$, where $H$ is a closed subgroup, is \emph{horospherical} if $H$ contains a maximal unipotent subgroup of $G$. These spaces are locally trivial torus fiber bundles $G/H \to G/P$ 
over flag varieties, where $P$ is the parabolic subgroup normalizing $H$. The fibers are isomorphic to the algebraic torus $T = P/H$. 
The character lattice of $T$ and the set of simple roots indexing the Schubert divisors in $G/P$ entirely describe $G/H$, see Section~\ref{s:horo-def}.

A \emph{horospherical $G$-variety} is a normal $G$-variety such that every orbit is horospherical. Its \emph{complexity} 
is the minimum of the codimensions of its $G$-orbits. 
Horospherical varieties of complexity at most one admit a combinatorial description in terms of objects coming from convex geometry, almost analogous to the classical situation for toric varieties. This combinatorial description was introduced by Timashev in \cite{Timashev:Gmanifolds1997}, and adapted from Luna--Vust theory \cite[Chapter 3]{TimashevBook}, see also \cite{LangloisTerpereau}. 
It goes as follows.

Consider a \emph{simple $G$-variety} $X$, i.e., a $G$-variety which contains an open affine $B$-stable subset which intersects every $G$-orbit of $X$. Any simple horospherical $G$-variety $X$ of complexity one is described by a $4$-tuple $(C,G/H,\D,\F)$, where $C$ is a smooth projective curve and $G/H$ is a horospherical $G$-homogeneous space. The pair $(C,G/H)$ encodes the birational equivariant type of $X$. The third datum $\D$ is a \emph{polyhedral divisor} on $C$, i.e., a Weil divisor on $C$ whose coefficients are polyhedra, see Definition~\ref{d:polyhedral-divisor}. Finally $\F$ denotes the set of $B$-stable prime divisors on $X$ which are not $G$-stable and which contain a $G$-orbit of $X$. We call the pair $(\D,\F)$ a \emph{colored polyhedral divisor}. 

It follows from \cite{Sumihiro:EquivCompletion1} that any normal $G$-variety possesses a finite open covering by stable subsets which are simple $G$-varieties. Thus any horospherical $G$-variety of complexity one can be described by a triple $(C,G/H,\DF)$, where $\DF$ is a \emph{colored divisorial fan} consisting in a certain finite collection of colored polyhedral divisors, see Definition~\ref{d:polyhedral-divisor}. 

The goal of this paper is to compute the stringy motivic volume of $X$ in terms of $(C,G/H,\DF)$, under the assumption that $X$ is a normal $\QQ$-Gorenstein variety with log terminal singularities. 
For such $X$, let $f : X' \to X$ be a log resolution of singularities. The \emph{stringy motivic volume} of $X$ is the integral
\begin{equation}\label{e:motivicvolume}
	\E(X) = \int_{\L(X')} \LL^{-\ord K_{X'/X}}d\mm_{X'}.
\end{equation}
Here $\mathbb L$ is the class of the affine line, 
$K_{X'/X}$ the relative canonical divisor, $\L(X')$ the \emph{arc space} of $X'$, and 
the measure $\mm_{X'}$ takes values in a modification of the Grothendieck ring $\KO$ of complex algebraic varieties.

The measure $\mm_{X'}$ can be explicitly described on particular subsets of $\L(X')$. Denote by $\pi_{n}$ the truncation morphism sending an arc of $X'$ to its $n$-jet. A subset $Z \subset \L(X')$ is a \emph{cylinder} if there is $n_0 \in \NN$ such that for any $n\geq n_0$, the subset $\tronc_n(Z)$ is a constructible set in the space of $n$-jets of $X'$, and $Z = \tronc_n^{-1}(\tronc_n(Z))$. On the cylinders the motivic 
volume is defined as
\[
	\mm_{X'}(Z) = [\tronc_n(Z)] \LL^{-n\dim X'}
\]
and does not depend on the choice of $n\geq n_0$. 
Measurable subsets of $\L(X')$ are approximations of cylinders. 
The integral of a measurable function $F : \L(X') \to \ZZ$ with respect to $\mm_{X'}$ is the sum (if it exists) of the series $\sum_{s\in\ZZ} \mm_{X'} (F^{-1}(s)) \LL^{-s}$. 

Note that the stringy motivic volume has a rational form (see the alternative expression in Definition~\ref{d:vol-mot-rational}) and it is independent of the choice of a resolution of singularities, see \cite{Batyrev1998:StringyHodge}.

If $X$ is toric with lattice $N$ and fan $\Sigma$, 
let  $\theta_X : |\Sigma| \to \QQ$  be the piecewise linear 
function which takes value $-1$ 
on each primitive generator of the one-dimensional cones of $\Sigma$. 
Here $|\Sigma|$ denotes the reunion of all cones of $\Sigma$. 
Then
\begin{equation}\label{e:stringyinvariant}
	\E(X) = (\LL-1)^r \sum_{\nu \in |\Sigma|\cap N} \LL^{\theta_X(\nu)}.
\end{equation}
Equation~\eqref{e:stringyinvariant} was proved in \cite[Theorem 4.3]{Batyrev-Moreau}, albeit in a more general setting. We follow the proof of \emph{loc. cit.} which consists of two main steps. 

Let $K$ be a field extension of $\CC$. The set $\L(X)(K)$ of $K$-points in the arc space $\L(X)$ of the toric variety $X$ identifies with $X(\O)$, where $\O = \O_K :=  K[[t]]$. 
The torus $T(\O)$ naturally acts on $X(\O) \cap T(\K)$, where $\K$ is the fraction field of $\O$, and $X(\O),T(\K)$ are viewed as subsets of $X(\K)$. 
The set of orbits of $X(\O) \cap T(\K)$ is in bijection with $|\Sigma|\cap N$  \cite{Ishii:ArcSpaceToric}. This yields 
natural cylinders $\C_\nu \subset \L(X)$ for any $\nu \in |\Sigma|\cap N$.

In the second step, we study the integral $\int \LL^{-\ord K_{X'/X}}d\mm_{X'}$ along each cylinder $\C_\nu$, where $X' \to X$ 
is a desingularisation given by a fan subdivision. Since $X'(\O) \cap T(\K)$ identifies with $X(\O) \cap T(\K)$ and the complement in $\L(X')$ has motivic measure zero, we obtain an equality
\begin{equation}\label{e:integralequality}
	\int_{\L(X')} \LL^{-\ord K_{X'/X}}d\mm_{X'} = \sum_{\nu \in |\Sigma|\cap N} \int_{\C_\nu} \LL^{-\ord K_{X'/X}}d\mm_{X'}.
\end{equation}
Equation~\eqref{e:stringyinvariant} is then obtained by computing each summand of the right-hand side.

When adapting these two steps to horospherical varieties of complexity one we encounter significant differences. 
For simplicity, let us explain these differences in the complexity one \emph{toric} case, i.e., assume that $X$ is a normal $\QQ$-Gorenstein $T$-variety 
of complexity one with log terminal singularities. We denote by $(C,T,\DF)$ the combinatorial data associated with $X$. The first step is then modified as follows. 
The variety $X$ contains an open subset of the form $\Gamma \times T$, where $\Gamma \subseteq C$ is an open dense subset. 
Thus comparing with the toric setting it is natural to consider the action of $T(\O)$ on
\begin{equation}\label{e:arcspace}
	X(\O) \cap (\Gamma \times T) (\mathcal{K}).
\end{equation}
However, when we follow the almost similar approach of Ishii \cite{Ishii:ArcSpaceToric}, 
we find uncountably many $T(\O)$-orbits. To solve this problem, 
we cut the set \eqref{e:arcspace} in two disjoint pieces, namely, the subset of \emph{horizontal arcs} and the subset of \emph{vertical arcs}. Thus we obtain a partition of the set \eqref{e:arcspace} by parts $\C_{(y,\nu,\ell)}$ which are $T(\O)$-stable and indexed by a countable set $|\DF|_{\Gamma} \cap \N$ combinatorially determined by $\DF$.

The second step consists in checking that each $\C_{(y,\nu,\ell)}$ is measurable, see Lemmas~\ref{l:vol-hor} and~\ref{l:vol-ver}, and to compute
\begin{equation}\label{e:integralrelation}
	\int_{\L(X')} \LL^{-\ord K_{X'/X}}d\mm_{X'} = \sum_{(y,\nu,\ell) \in |\DF|\cap \N} \int_{\C_{(y,\nu,\ell)}} \LL^{-\ord K_{X'/X}}d\mm_{X'},
\end{equation}
where $X'$ is obtained by an explicit desingularization determined by $\DF$, see Section~\ref{s:desingularization}. 

The function $\theta_X : |\Sigma| \to \QQ$ introduced above in the toric setting is a support function of a torus-invariant canonical divisor. This analogy between support functions and Cartier $\QQ$-divisors persists for horospherical varieties of complexity one, see Section~\ref{s:support}.
However it is not possible, adapting the proofs in \cite{Batyrev-Moreau}, to directly obtain a version involving the support function of an invariant canonical divisor. Instead we need to introduce an \emph{auxiliary function} $\omega_X$, see Proposition-Definition~\ref{pd:support}. 
This $\omega_X$ does not define in the usual sense a Cartier $\QQ$-divisor
(see Section~\ref{s:hypersurface} for further details). Note that in \cite{Lan17} we use $\omega_X$ for characterizing singularities related
to the minimal model program.

Our main result can be stated as follows. 
\begin{theorem}\label{t:main}
Let $X$ be a $\QQ$-Gorenstein horospherical variety of complexity one with log terminal singularities and 
combinatorial data $(C,G/H,\DF)$. 
Assume that $\Gamma\subset C$ is a dense open subset which does not contain any special point (see Definition~\ref{d:specialpoints}). Then 
\[
	\E(X) = [G/H]\left[ \sum_{(y,\nu,\ell) \in |\DF|_{\Gamma}\cap \N} [X_\ell] \,\LL^{\omega_X(y,\nu,\ell)}  \right],
\]
where $X_0 = \Gamma$ and $X_\ell = \AA^1 \setminus \{0\}$ if $\ell \geq 1$.
\end{theorem}

Let us give a brief summary of the contents of each section. In the second and third sections, we introduce notations from motivic integration and horospherical varieties that we use throughout this paper. In the fourth section, we study the arc space of a horospherical variety of complexity one. To do this we proceed in two steps. 
First, we give a description in the case where the acting group is an algebraic torus, this corresponds to Sections~\ref{s:hor-ver} and~\ref{s:T-non-affine}. 
The last step concerns the general case, where we reduce to the case of torus actions by parabolic induction. 
We believe that the study in the first step, independently, might be interesting for the study of singularities of $T$-varieties. 

In the fifth section, we prove Theorem~\ref{t:main}. One of the main ingredients of the proof is the construction by combinatorial methods of an explicit desingularization of $X$. 
This construction is explained in Section~\ref{s:desingularization}. Then we give a precise description of the rational form of the stringy motivic volume in terms 
of $\omega_{X}$.
 In the last section, we provide examples and applications. In particular, for a certain class of log terminal horospherical varieties of complexity one we give a simple expression of another invariant called the \emph{stringy Euler characteristic}, see Lemma~\ref{l:stringy-euler}. 
This allows us to rephrase for this class the smoothness condition in terms of the stringy Euler characteristic, see Theorem~\ref{t:smoothness}. 

\section{Arc spaces and motivic integration}
\label{sect:2}
Motivic integration was developed by Batyrev in the smooth case \cite{Batyrev1998:StringyHodge} and by Denef and Loeser \cite{DenefLoeser99:Arc} in full generality. 
In this section, we recall main facts on this theory used throughout the paper (see also \cite{Looijenga, Loeser:Seattle, Blickle2011:Motivic} for survey articles).

\subsection{The arc space of a variety}\label{s:arcspace-def}
Let $X$ be a variety. For any integer $m$ we denote by $\L_m(X)$ the $m$-jet scheme of arcs of $X$. It is a scheme of finite type, and for any field extension $K$ of $\CC$,  its $K$-rational points are the $K[t]/(t^{m+1})$-rational points of $X$. We  will consider it with its reduced structure.
For any $m \geq \ell \geq 0$ we denote by $j^{m,l}:\L_m(X)\to\L_\ell(X)$ the transition morphism induced by the Weil restriction morphism 
from $\CC[t]/(t^{\ell+1})$ to $\CC[t]/(t^{m+1})$. The \emph{arc space} of $X$
is the limit of the projective system $(\L_m(X),(j^{m,l}))$. It is a scheme and for any  field extension $K$ of $\CC$ each $K$-rational point of $\L(X)$ corresponds to a unique $K[[t]]$-rational point of $X$, and vice-versa.
We write  $\tronc_\ell :\L(X) \to \L_\ell(X)$ for the natural truncation morphisms. They commute with the transition morphisms.

\subsection{The Grothendieck ring $\KO$}\label{s:KO} \label{d:grothendieck-ring}
We denote by $[S]\in \KO$ the class of an algebraic variety $S$, and let $\LL = [\AA^1_\CC]$. 
We write $\Mc$ for the completion of $\KO[\LL^{-1}]$ with respect to the usual filtration 
ensuring that $\mathbb L^{-n}\to 0$ (see \cite{DenefLoeser99:Arc}, \cite{Blickle2011:Motivic} and \cite{Nicaise-Sebag:K0}). 

\subsection{Motivic integration}\label{s:motint}

Let $X$ be a smooth $d$-dimensional variety. 

\begin{definition} A subset $C\subset \L(X)$ is a \emph{cylinder} if $C=\tronc_{m_0}^{-1}(\tronc_{m_0}(C))$, for some $m_0$, 
and the set $\tronc_{m_0}(C)$, called the \emph{$m_0$-basis} of $C$, is constructible in $\L_{m_0}(X)$. Note that $\L(X)$ itself is a cylinder (see~\cite[Corollary 2]{Greenberg1966:henselian}). The motivic measure of $C$ is defined as
\[
	\mm_X(C) = [\tronc_m(C)]\LL^{-md}, \quad \text{ for any} \quad m \geq m_0.
\]
This does not depend on $m \geq m_0$ since if $m\geq \ell \geq 0$, the map $j^{m,\ell}\mid_{\tronc_m(C)}: \tronc_m(C)\to \tronc_\ell(C)$ 
is a piecewise trivial fibration with fiber $\AA^{(m-\ell)d}$.
\end{definition} 
\begin{remark}\label{r:measure-closed}
	A larger family of measurable sets can be considered (see~ \cite{Batyrev1998:StringyHodge, DenefLoeser99:Arc}, \cite[Appendix]{Denef-Loeser:McKay}). For instance, if $Y \subsetneq X$ is a closed reduced subscheme, then $\L(Y)$ is measurable, and $\mm_X(\L(Y))=0$ (see~\cite[Equation (3.2.2)]{DenefLoeser99:Arc} or \cite[Proposition 4,5]{Blickle2011:Motivic}).
\end{remark}

\begin{definition}
	A function $F : \L(X) \to \ZZ \cup \{ \infty \}$ is \emph{measurable} if for any $s \in \ZZ \cup \{ \infty \}$ the subset $F^{-1}(s)$ is measurable. In that case, 
\emph{$\LL^{-F}$ is integrable} over a measurable set $C$ of $\L(X)$ if $\mm_X(F^{-1}(\infty))=0$ and the series 
$\sum_{s \in \ZZ} \mm_X(C \cap F^{-1} (s)) \LL^{-s}$
converges in $\Mc$. Then, we let
\[
	\int_C \LL^{-F} d\mm_X := \sum_{s \in \ZZ} \mm_X(C \cap F^{-1} (s)) \LL^{-s}.
\]
\end{definition}

\begin{example}[{\cite[Section 2.4]{Blickle2011:Motivic}}]
	Let $Y$ be a proper closed subscheme of $X$ defined by the sheaf of ideals $\I_Y$.  
	Let $\gamma$ be a point in $\L(X)$ and $\gamma^* : \OO_X \to \OO_{\Spec(k_\gamma[[t]])}$ be its sheaf homomorphism. The \emph{order of $\gamma$ along $Y$} is  $\ord_Y(\gamma):=\sup \left\{ e \in \NN \cup \{ \infty \} \mid \gamma^*(\I_Y) \subseteq (t^e) \right\}$.
	The function $\ord_Y$ is measurable and the set $\ord_Y^{-1}(\infty)$, equal to $\L(Y)$, has zero measure.
	Thus, for a divisor $D=\sum_{i=1}^s a_i D_i$ in $X$ (where $a_i \in \QQ$ and the $D_i$ are prime divisors) and for an arc $\gamma \in \L(X)$, we set 
$\ord_D(\gamma)= \sum_{i=1}^s a_i \ord_{D_i}(\gamma)$
with, value $\infty$ if one of the $\ord_{D_i}(\gamma)$ is infinite. The function $\ord_D$ is then measurable.
\end{example}

\subsection{Stringy invariants of varieties with log terminal singularities}\label{s:def-stringy}

Let $X$ be an irreducible normal $\QQ$-Gorenstein variety, namely, $X$ is a normal variety and a (thus any) canonical divisor $K_X$ is $\QQ$-Cartier. We say 
that a morphism $f : X' \to X$ is a \emph{resolution of singularities} if $X'$ is smooth,  $f$ is birational and proper, and the exceptional locus is the reunion of finitely many normal crossings irreducible smooth divisors. 
\begin{definition} \label{d:vol-mot-rational}
	Let $X$ be a normal $\mathbb Q$-Gorenstein algebraic variety with log terminal singularities. Let $f : X' \ra X$ be a resolution of singularities of $X$ and let $(E_i)_{i\in I}$ be the set of irreducible components of its exceptional locus. 
	Let $K_X$ (resp. $K_{X'}$) be a canonical divisor of $X$ (resp. of $X'$) and $(\nu_i)_{i\in I}$ be the multiplicities of the relative canonical divisor 
	$$ K_{X'/X} := K_{X'} - f^* K_X = \sum_{i \in I} \nu_i E_i. $$
	The \emph{stringy motivic volume $\E(X)$} is usually defined following \cite{Batyrev1998:StringyHodge} as 
	\[
		\E(X) := \sum_{J \subset I} [E_J^\circ] \prod_{j \in J} \frac{\LL-1}{\LL^{\nu_j+1}-1} \in \Mc(\LL^{\frac{1}{m}}),
	\]
	where $m$ is the g.c.d of the denominators of the $\nu_i$, the ring $\Mc(\LL^{\frac{1}{m}})$ is defined as the localization and completion of $\KO$  with respect to $\LL^{\frac{1}{m}}$, and $E_J^\circ$ is $\bigcap_{j \in J} E_j \setminus \bigcup_{i \in I \setminus J} E_i$ for every subset $J$ of $I$. 
\end{definition}

The theorem below follows from the appendix of \cite{Batyrev1998:StringyHodge}, see also \cite{Denef-Loeser:McKay} and \cite{Yasuda} :
\begin{theorem}\label{t:def-stringy}
	The stringy motivic volume is equal to
	\[
		\E(X) = \int_{\L(X')} \LL^{-\ord_{K_{X'/X}}} d \mm_{X'} \in \Mc(\LL^{\frac{1}{m}})
	\]
	and does not depend on the resolution $f : X' \to X$. If $X$ is smooth, then $\E(X) = [X]$.
\end{theorem}

From $\E(X)$ we can deduce two new invariants for log terminal varieties. These invariants are defined by Batyrev (\cite{Batyrev1998:StringyHodge}) and also studied by several authors such as Denef--Loeser 
\cite{Denef-Loeser:McKay}, Yasuda \cite{Yasuda} and Schepers--Veys (\cite{Veys:Stringyzeta, SchepersVeys:stringy}).
\begin{definition}\label{d:E-function}
	The \emph{$E$-function} of a log terminal variety $X$ is defined as
	\[
		\EE(X;u,v) := \sum_{J \subset I} E(E_J;u,v) \prod_{j \in J} \frac{uv -1}{(uv)^{\nu_j+1}-1},
	\]
	where for any variety $V$, $E(V;u,v)$ 	is the usual Hodge-Deligne polynomial of $V$
	in the variables $u$ and $v$.
	The \emph{stringy Euler characteristic} of $X$ is
	\[
		e_{st}(X) := \sum_{J \subset I} e(E_J) \prod_{j \in J} \frac{1}{\nu_j+1},
	\]
	where $e(E_J)=E(E_J;1,1)$ denotes the usual Euler characteristic.
\end{definition}

\section{Horospherical actions}


Throughout, 
let $G$ be a connected simply connected reductive algebraic group over $\CC$ with Borel subgroup $B$, maximal torus $\tor$, and maximal unipotent subgroup $U$ such that $B=\tor U$. 

\subsection{Horospherical transformations}\label{s:horo-def}

An ($G/H$-)\emph{embedding} of a horospherical homogeneous space $G/H$ is a pair $(X,x)$ such that $X$ is a normal $G$-variety containing $x$, the orbit $G \cdot x$ of $x$ is open, and the stabilizer $G_x$ of $x$ is $H$. For brevity, we denote such an embedding simply by $X$. The subgroup $H$ is called a \emph{horospherical subgroup}.

\paragraph{{\bf Description of horospherical subgroups of $G$.}} Let $\Phi$ be the set of simple roots of $G$ and $W=N_G(\tor)/\tor$ be the Weyl group of $G$. We denote by $s_\alpha \in W$ the simple reflection associated with $\alpha \in \Phi$ and by $\dot s_\alpha$ a lift of $s_\alpha$ to $G$. 
We also denote by $w_0$ the longest element of $W$. If $I \subset \Phi$ we set $\dot W_I := \langle \dot s_\alpha \mid \alpha \in I \rangle$ and we let $P_I$ be the parabolic subgroup generated by $B$ and $\dot W_I$. The map $I \mapsto P_I$ between subsets of $\Phi$ and closed subgroups of $G$ containing $B$ is bijective.

By \cite[Proposition 2.4]{Pasquier2008:VarHoroFano}, there is a bijection between subgroups $H$ of $G$ containing $U$ and pairs $(M,I)$, where $M$ is a sublattice of the character group $\XX(\tor)$, and $I \subset \Phi$, subject to the property that for any $\alpha \in I$, the associated coroot $\alpha^\vee$ satisfies $\langle m,\alpha^\vee\rangle = 0$ for any $m \in M$. The bijection is constructed as follows.
First consider the normalizer $N_G(H)$ of $H$ in the group $G$. It is a parabolic subgroup $P=P_I$, where $I \subset \Phi$. The subset $I$ constitutes the first part of the data. Now the quotient $T := P/H$ is a torus, and our second datum $M$ is the character lattice of $T$.

\paragraph{{\bf Colors and charts.}} Let $Y$ be a horospherical $G$-variety of arbitrary complexity and $L \subset G$ be a closed subgroup of $G$. An \emph{$L$-divisor} is an $L$-stable prime divisor, and a \emph{color} is a $B$-divisor which is not $G$-stable.

A \emph{$B$-chart} of $Y$ is a dense open $B$-stable affine subset of $Y$. We say that the horospherical variety $Y$ is \emph{simple} if $Y$ possesses a $B$-chart which intersects every $G$-orbit of $Y$. Since $Y$ is normal, it is the reunion of simple $G$-stable open subsets, see \cite[Theorem 1]{Sumihiro:EquivCompletion1} and \cite[Theorem 1.3]{Knop:Luna-Vust}. Moreover every normal simple horospherical $G$-variety is quasi-projective.

\paragraph{{\bf Models.}} Let $Y$ be a variety with a horospherical $G$-action. We say that a $G$-variety $Z$ is a \emph{model} of $Y$ if $Z$ is normal and there exists a $G$-equivariant birational map $Z \dashrightarrow Y$. For instance, any horospherical $G$-variety $Y$ has a model  of the form $C \times G/H$, where $H$ is a horospherical subgroup of $G$, $C$ is a smooth variety, and $G$ acts by left multiplication on $G/H$ and trivially on $C$ (see \cite[Proposition 7.7]{TimashevBook}). This implies that the dimension of $C$ is equal to the complexity of the action of $G$ on $Y$ .  

Let $Y$ be a horospherical $G$-variety with model $Z = C \times G/H$. The variety $Y$ contains a $G$-stable open subset which is identified with $\Gamma \times G/H \subset Z$, where $\Gamma \subset C$ is open dense. Let $(M,I)$ be the lattice/subset pair parameterizing $H$. To any $\alpha \in \Phi \setminus I$ we associate a color by letting $D_\alpha$ denote the set $\overline{\Gamma \times p^{-1}(X_\alpha)} \subset Y$, where $p : G/H \to G/P$ is the natural projection, $P = N_G(H)$, and $X_{\alpha}$ is the corresponding Schubert variety. The map $\alpha \mapsto D_{\alpha}$ between $\Phi \setminus I$ and the set $\F_Y$ of colors of $Y$ is bijective, see \cite[Lemma 3]{FMSS95}. 
We will use the same notation $D_\alpha$ for the colors of $G/H$ and the colors of $Y$.

\subsection{Colored polyhedral divisors}\label{s:complexity-one}

In this subsection we consider horospherical varieties of complexity one, which as we have seen are $G$-equivariantly birational to direct products $Z = C \times G/H$, where $C$ is a smooth projective curve and $H$ is a horospherical subgroup of $G$. Here we introduce some combinatorial data which describes all the models of $Z$. We refer to \cite[Section 16]{TimashevBook}.

Let $(M,I)$ be the pair associated with the horospherical subgroup $H$, and $N$ be the dual lattice $\Hom(M,\ZZ)$. Denote by $M_\QQ$ and $N_\QQ$ the dual associated vector spaces  $\QQ \otimes_\ZZ M$ and $\QQ \otimes_\ZZ N$. The \emph{dual} $\sigma^\vee \subset M_\QQ$ of a polyhedral
cone $\sigma \subset N_\QQ$ is
\[
	\sigma^\vee = \left\{ m \in M_\QQ \mid \forall v \in \sigma,\, \langle m,v \rangle \geq 0 \right\}.
\]
The polyhedral cone $\sigma$ is \emph{strongly convex} (i.e., $\{0\}$ is a face of $\sigma$) if and only if $\sigma^\vee$ generates $M_\QQ$.
\begin{definition}[\cite{AltmannHausen}]\label{d:polyhedral-divisor}
	Let $\sigma$ be a strongly convex polyhedral cone in $N_\QQ$ and $C_0$ be a smooth curve.
	\begin{enumerate}
		\item A \emph{$\sigma$-polyhedral divisor} on $C_0$ is a formal sum
			\[
				\D = \sum_{y \in C_0} \Delta_y \cdot [y],
			\]
			where $\Delta_y \subset N_\QQ$ is a \emph{$\sigma$-polyhedron} (that is, the Minkowski sum of the cone $\sigma$ and a non-empty polytope $\Pi \subset N_\QQ$), with the property that $\Delta_y=\sigma$ for all but finitely many $y \in C_0$. The 	\emph{set of special points} of $\D$ is $\Sp (\D) := \{ y \in C_0 \mid \Delta_y \neq \sigma \}$.
			The cone $\sigma$ is called the \emph{tail} of $\D$ and the curve $C_0$ is its \emph{locus}.
		\item The \emph{degree} of $\D$ is defined, if $C_0$ is complete, as the Minkowski sum $\deg \D := \sum_{y \in C_0} \Delta_y$ inside $N_\QQ$, and as the empty set otherwise.
		\item For $m \in \sigma^\vee$ we define a $\QQ$-divisor on $C_0$, called the \emph{evaluation of $\D$ at $m$,} by setting
			\[
				\D(m) := \sum_{y \in C_0} \min \{ \langle m,v \rangle  \mid v \in \Delta_y \}  \cdot [y].
			\]
	\end{enumerate}
\end{definition}

\begin{definition}\label{d:Malgebra}
\begin{enumerate}
\item For any subsemigroup $S$ of $M$ we let $\CC[S] = \bigoplus_{m \in S} \CC \chi^m$ be the $\CC$-algebra subject to the relations $\chi^m \cdot \chi^{m'} = \chi^{m+m'}$ for all $m,m' \in S$.
\item Let $\D$ be a polyhedral divisor. The \emph{associated $M$-graded algebra} of $\D$ is the subalgebra
			\[
			A(C_0,\D) := \bigoplus_{m \in \sigma^\vee \cap M} H^0(C_0,\OO_{C_0}(\lfloor  \D(m) \rfloor)) \otimes \chi^m \subset \CC(C_0) \otimes_\CC \CC[M],
		\]
		where $\lfloor  \D(m) \rfloor$ is obtained by taking the integral part of the coefficients of $\D(m)$.
	\end{enumerate}
\end{definition}

\begin{definition}
	 A $\sigma$-polyhedral divisor $\D$ is \emph{proper} if either $C_0$ is affine, or it is projective and the two following conditions are satisfied:
		\begin{enumerate}[(i)]
			\item $\deg \D \subsetneq \sigma$;
			\item If $\min \{ \langle m,v \rangle  \mid v \in \deg \D \} = 0$ for some $m \in \sigma^\vee$, then $\D(rm)$ is principal for some $r \in \ZZ_{>0}$.
		\end{enumerate}
	If $\D$ is proper then $A(C_0,\D)$ and $\CC(C_0) \otimes_\CC \CC[M]$ have the same field of fractions.
\end{definition}
The next definition introduces the \emph{coloration map} $\varrho$ of the horospherical homogeneous space $G/H$.
\begin{definition}
	Let $(M,I)$ be the lattice/subset pair associated with the horospherical subgroup $H$. For any color $D$ of $G/H$ there exists a unique root $\alpha \in \Phi$ such that $D = D_\alpha$. In particular, the coroot $\alpha^\vee$ is an element of $\Hom(\XX(\tor),\ZZ)$. Thus its restriction to $M$, $\alpha^\vee_{\mid M} : M \to \ZZ$, is an element of $N = \Hom(M,\ZZ)$, and we define $\varrho(D)$ as $\alpha^\vee_{\mid M}$. 
\end{definition}

\begin{definition}
	A pair $(\D,\F)$ is a \emph{colored $\sigma$-polyhedral divisor} (see~\cite[Section 1.3]{LangloisTerpereau}) on $C_0$ if $\F$ is a subset of $\F_{G/H}$ such that $\varrho(\F) \subset \sigma$ and $0 \not \in \varrho(\F)$, and if $\D$ is a proper $\sigma$-polyhedral divisor on $C_0$.
\end{definition}

We now give a combinatorial description of the simple models of the product $C \times G/H$, which follows from a general classification result for normal $G$-varieties of complexity one, see~\cite[Theorem 3.1]{Timashev:Gmanifolds1997}.

\begin{theorem}\label{t:classification-horo}
	Let $G$ be a connected simply connected reductive algebraic group and $H$ be a horospherical subgroup associated with a lattice/subset pair $(M,I)$. Define $Z:=C \times G/H$, where $C$ is a smooth projective curve.
	\begin{enumerate}[(i)]
		\item Let $(\D,\F)$ be a colored polyhedral divisor on an open dense subset $C_0 \subset C$. Then there exists a simple $G$-model $X(\D) := X(\D,\F)$ of $Z$ containing a $B$-chart which is determined by $(\D,\F)$. The construction of $X(\D)$ will be specified in Equation~\eqref{e:XD}.
		\item Let $X$ be a simple $G$-model of $Z$. Then there exists an open dense subset $C_0 \subset C$, a colored polyhedral divisor $(\D,\F)$ on $C_0$, and a $G$-isomorphism $X \to X(\D)$.
	\end{enumerate}
\end{theorem}

Let us sketch the construction of the $G$-variety $X(\D)$ from item $(i)$. Denote by $P_\F$ the parabolic subgroup of $G$ containing the Borel subgroup $B$ and corresponding to the set of roots
\[
	I_\F = \{ \alpha \in \Phi \setminus I \mid D_\alpha \in \F \} \cup I.
\]
Let $L$ be the Levi subgroup of $P_\F$ containing the maximal torus $\tor$ and $B_L$ be the Borel subgroup of $L$ containing $\tor$, such that $I_\F$ is the set of simple roots of $L$. We also write $H_L:= H \cap L$ and $U_L:=U \cap L$. The subgroup $H_L$ contains $U_L$, hence the homogeneous space $L/H_L$ is horospherical. It is quasi-affine, see~\cite[Corollary 15.6]{TimashevBook}. Note that $L/H_L$ is constructed so that the set of $B$-divisors in $L/H_L$ exactly identifies with $\F$. Moreover the lattice of $B_L$-weights in $\CC(L/H_L)$ is the lattice $M$. 

Denote by $\dw \subset M_\QQ$ the cone generated by the dominant weights of $\CC[L/H_L]$. Recall that the rational $L$-module $\CC[L/H_L]$ admits a $\CC$-algebra graduation 
\[
	\CC[L/H_L] = \bigoplus_{m \in \dw \cap M} V(m),
\]
where $V(m)$ denotes the irreducible rational representation corresponding to $m \in \dw \cap M$ (see~ \cite[Proposition 7.6]{TimashevBook}). 

Since by definition of $(\D,\F)$ we have $\varrho(\F) \subset \sigma$, it follows that $\sigma^\vee \subset \dw$. Thus we may define a subalgebra of $\CC(C) \otimes_\CC \CC[L/H_L]$ by
\begin{equation}\label{e:ACDF}
	A(C_0,\D,\F) := \bigoplus_{m \in \sigma^\vee \cap M} \Am \otimes_\CC V(m)
\end{equation}
Set $Y(\D,\F) := \Spec A(C_0,\D,\F)$. Since $\D$ is proper, $Y(\D,\F)$ is a well-defined affine $L$-model of the product $C \times L/H_L$. 
To conclude, consider the parabolic induction:
\begin{equation}\label{e:XD}
	X(\D) = X(\D,\F) := G \times^{P_\F} Y(\D,\F).
\end{equation}
The notation $G \times^{P_\F} Y(\D,\F)$ means that $X(\D)$ is the quotient $(G \times Y(\D,\F))/P_\F$, where $P_\F$ acts by
	 $p \cdot (g,y) := (gp^{-1},\phi(p) \cdot y)$
for all $p \in P_\F$, $g \in G$, and $y \in Y(\D,\F)$. Here $\phi : P_\F \to L$ is the usual projection. Note that $X(\D)$ is a fiber bundle over $G/P_\F$ with fiber $Y(\D, \F)$.

The original construction of the simple $G$-models $X(\D)$ from \cite{Timashev:Gmanifolds1997} generalized the approach of \cite{Knop:Luna-Vust} in complexity one. Let us recall here how this construction relates to the one explained above.


\begin{definition}
	Let $Y$ be a horospherical $G$-variety. A \emph{locality} of $\CC(Y)$ is a local ring $R_{\p}$, where $\p \subset R$ is a prime ideal in a subalgebra $R$ of $\CC(Y)$ of finite type with fraction field $\CC(Y)$. The set of localities $\Sc(Y)$ is naturally endowed with a $\CC$-scheme structure, and it is called the \emph{scheme of geometric localities}. 
We denote by $\Sc_G(Y) \subset \Sc(Y)$ the maximal normal open subset where the $G$-action (induced by the $G$-action on $\CC(Y)$) is a morphism of $\CC$-schemes; we will call it the \emph{equivariant scheme of geometric localities}. 
\end{definition}

According to the previous definition, a $G$-model of a horospherical $G$-variety $Y$ is nothing but a $G$-stable separated open dense subset of $\Sc_G(Y)$ which is also of finite type over $\CC$.

Now consider a horospherical $G$-variety of the form $Z = C \times G/H$, where $C$ is a smooth projective curve and $H$ is a horospherical subgroup. 

\begin{definition}\label{d:hyperspace}
	The \emph{hyperspace} $\N_\QQ$ is the set $C\times N_\QQ \times \QQ_{\geq 0}$ modulo the equivalence relation
	\[
		(y,\nu,\ell) \sim (y',\nu',\ell') \Leftrightarrow (y=y',\nu=\nu',\ell=\ell') \text{ or } (\ell=\ell'=0,\nu=\nu').
	\] 
	Since an element $(y,\nu,0) \in \N_\QQ$ does not depend on $y \in C$, in the sequel we will denote it by $(C,\nu,0)$. 
	We also write $\N \subset \N_\QQ$ for the subset of integral points
	\[
		\N = \{\, (y,\nu,\ell) \in \N_\QQ \mid \nu \in N,\ell \in \ZZ_{\geq 0} \, \}.
	\]
	Let $\D = \sum_{y \in C_0} \Delta_y \cdot [y]$ be a $\sigma$-polyhedral divisor on an open dense subset $C_0 \subset C$. For $y \in C_0$, the \emph{Cayley cone} $C_y(\D)$ is defined as the cone generated by the subsets $\sigma \times \{0\}$ and $\Delta_y \times \{1\}$ in $N_\QQ \times \QQ_{\geq 0}$. The \emph{hypercone} associated with $\D$ is the subset $C(\D) \subset \N_\QQ$ defined by
	\[
		C(\D) = \left.\raisebox{.2em}{$\bigcup_{y \in C_0} \{ y \} \times C_y(\D)$} \middle/ \raisebox{-.2em}{$\sim$} \right. .
	\]
	We denote by $C(\D)(1)$ the subset of elements $[y,\nu,\ell] \in C(\D) \cap \N$ such that $(\nu,\ell)$ is an integral primitive generator of a one-dimensional face of $C_y(\D)$.
\end{definition}

Each triple $\xi=(y,\nu,\ell) \in \N_\QQ$ induces a unique discrete $G$-invariant valuation $w(\xi) : \CC(Z)^* \to \QQ$, whose restriction to the subalgebra $\CC(C) \otimes_\CC \CC[M] \subset \CC(Z)$ satisfies
\[
	w(\xi)(f\otimes \chi^m) = \ell \ord_y (f) + \langle m,\nu \rangle
\]
for all $f \in \CC(C)^*$ and $m \in M$. Here we identify the algebra $\CC[M]$ with $\CC[B x_0]^U$, where $B x_0$ is the open $B$-orbit in $G/H$. From this description it follows that the set of $G$-invariant discrete valuations of $\CC(Z)$ with values in $\QQ$ is in one-to-one correspondence with the hyperspace $\N_\QQ$ (see~\cite[Corollary 19.13, Theorems 20.3 and 21.10]{TimashevBook}).

Let us now consider the morphism $X(\D) \to G/P_\F$ induced by projecting on the first factor in Equation~\eqref{e:XD}; its fibers are isomorphic to $Y(\D,\F)$. The inverse image of the open $B$-orbit $B y_0$ of $G/P_\F$ is a $B$-chart which we denote by $X_0$. This $B$-chart intersects all the $G$-orbits in $X(\D)$. Note that $B y_0$ is isomorphic to the affine space $\AA_\CC^r$, where $r = \dim G/P_\F$. Since the fibration restricted to $B y_0$ is trivial, the chart $X_0$ is identified with $Y(\D,\F) \times \AA_\CC^r$. The ring of functions on $X_0$ can be described as
\begin{equation}\label{e:coordX0}
	\CC[X_0] = \left(\CC(C) \otimes_\CC \CC[Bx_0] \right) \cap \bigcap_{\xi \in C(\D)(1)} \OO_{v(\xi)} \cap \bigcap_{D \in \F} \OO_{v_D},
\end{equation}
where $v_D$ is the valuation associated with the divisor $D \in \F$ and $\OO_{v_D}$, $\OO_{v(\xi)}$ are the valuation rings respectively of $v_D$ and $v(\xi)$.

\begin{remark}\label{r:BOrbit}
The variety $X(\D)$ is a reunion of $G$-orbits
\begin{equation}\label{e:orbitsX0}
	X(\D) = G \cdot X_0 = \{ g \cdot x \mid g \in G, x \in X_0 \} \subset \Sc_G(Z).
\end{equation}
Thus Equations~\eqref{e:coordX0} and \eqref{e:orbitsX0} yield another more conceptual construction of the $G$-variety $X(\D)$, see~\cite[Section 16]{TimashevBook}. From now on we identify $X(\D)$ with an open subset of $\Sc_G(Z)$.
\end{remark}
\subsection{Colored divisorial fans}\label{s:colored-fans}

Here we introduce the combinatorial objects describing horospherical $G$-varieties of complexity one. The idea is to consider `fans of colored polyhedral divisors' in order to describe these varieties as a gluing of simple $G$-models. 

\begin{definition}[{\cite[\S 1.3.4]{LangloisTerpereau}}]
	Let $C$ be a smooth projective curve and $G/H$ be a horospherical homogeneous space. A \emph{colored divisorial fan} on the pair $(C,G/H)$ is a finite set $\DF=\{ (\D^i,\F^i) \mid i \in J \}$ of colored polyhedral divisors, where 
	\[
		\D^i = \sum_{y \in C^{i}} \Delta_y^i \cdot [y].
	\]
	Here $C^{i}$ is an open dense subset of $C$, $\F^i \subset \F_{G/H}$, and $\D^i,\F^i$ are subject to the following conditions:
	\begin{enumerate}[(i)]
		\item For all $i,j \in J$, we have $(\D^i \cap \D^j,\F^i \cap \F^j) \in \DF$, where the intersection is given by
			\[
				\D^i \cap \D^j := \sum_{y \in C_{i,j}} (\Delta_y^i \cap \Delta_y^j) \cdot [y],
			\]	
			and the curve $C_{i,j}$ by the equality
				$C_{i,j} := \{ y \in C^{i} \cap C^{j} \mid \Delta_y^i \cap \Delta_y^j \neq \emptyset \}$.
		\item For all $i,j \in J$, $y \in C_{i,j}$, the polyhedron $\Delta_y^i \cap \Delta_y^j$ is a common face to $\Delta_y^i$ and $\Delta_y^j$.
		\item For all $i,j \in J$ we have
			\[
				\F^i \cap \F^j = \varrho^{-1}(\sigma_i \cap \sigma_j) \cap \F^i = \varrho^{-1}(\sigma_i \cap \sigma_j) \cap \F^j,
			\]	
			where $\sigma_i$, $\sigma_j$ are the respective tails of $\D^i$, $\D^j$.
			\item The intersection of the degree of $\D^{i}$ with the tail $\sigma_{ij}$ of $\D^{i}\cap \D^{j}$
is equal to the intersection of the degree of $\D^{j}$ with $\sigma_{ij}$.
	\end{enumerate}
\end{definition}
To any colored divisorial fan $\DF$ one can attach some combinatorial objects as follows.

\begin{definition}\label{d:specialpoints}
The set
\[
	H(\DF) := \{ (C(\D),\F) \mid (\D,\F) \in \DF \}
\]
is the associated \emph{colored hyperfan} (see \cite[Definition 16.18]{TimashevBook}). The \emph{support} of $\DF$ is defined as
\[
	|\DF| := \bigcup_{(\D,\F) \in \DF} C(\D) \subset \N_\QQ
\]
and the \emph{set of special points} as
\[
\Sp(\DF) := \bigcup_{(\D,\F) \in \DF} \Sp(\D).
\]
We refer to \ref{d:polyhedral-divisor} for the definition of $\Sp(\D)$. 
The \emph{tail fan} of $\DF$ is the fan $\Sigma(\DF)$ generated by all the tails $\sigma$ of polyhedral divisors $\D$ with $(\D,\F) \in \DF$.
\end{definition}

When $G=T$ is a torus, we write $\DF = \{ \D^i \mid i \in J \}$ instead of $\DF = \{ (\D^i,\emptyset) \mid i \in J \}$, and we recover the usual divisorial fans of \cite[Definition 5.2]{AHS}. The following result provides a full description of horospherical varieties of complexity one.

\begin{theorem}[{\cite[Theorems 12.13 and 16.19]{TimashevBook}}]
	Let $C$ be a smooth projective curve, $G/H$ be a horospherical homogeneous space, and $Z$ be the product $C \times G/H$. Denote by $\DF$ a colored divisorial fan on $(C,G/H)$. Then the open reunion
	\[
		X(\DF) := \bigcup_{(\D,\F) \in \DF} X(\D,\F)
	\]
	inside the equivariant scheme of geometric localities $\Sc_G(Z)$ is a $G$-model of $Z$. Conversely, every $G$-model of $Z$ arises in this way. Moreover $X(\DF)$ is a complete variety if and only if $|\DF| = \N_\QQ$.
\end{theorem}

Let us now recall the construction of the discoloration morphism associated with a colored divisorial fan from~\cite[Section 2.2]{LangloisTerpereau}.
This construction will allow us to provide specific equivariant desingularizations of $X(\DF)$. If $\DF$ is a colored divisorial fan, we denote by 
\[
	\DFd := \{ (\D,\emptyset) \mid (\D,\F) \in \DF \}
\]
its \emph{discoloration}, that is the colored divisorial fan obtained from $\DF$
by removing the colors. For any $(\D,\F)$ in $\DF$ (respectively $(\D,\emptyset)$ in $\DFd$) we consider $X_0$ (respectively $X_{dis}$) the $B$-chart associated with $(\D,\F)$ (respectively with $(\D,\emptyset$)). The inclusions of $\CC$-algebras $\CC[X_0] \subset \CC[X_{dis}]$ induce morphisms $X(\D,\emptyset) \to X(\D,\F)$ which glue into a birational proper $G$-equivariant morphism
\begin{equation}\label{e:discolor}
	\pi_{dis} : X(\DFd) \to X(\DF).
\end{equation}
Writing $P = P_I = N_G(H)$, $T=P/H$ and $M=\XX(T)$, we define the $T$-variety associated with $\DF$:
\begin{equation}\label{e:V(E)}
	V(\DF) := \bigcup_{i \in J} \Spec A(C^{i},\D^i) \subset \Sc_T(C \times T),
\end{equation}
where $\Sc_T(C \times T)$ is the $T$-equivariant scheme of geometric localities of $C \times T$, and $A(C^{i},\D^i)$ denotes the $M$-graded algebra
associated with $\D^i$, see Definition \ref{d:Malgebra}.
From \cite[Proposition 2.9]{LangloisTerpereau} it follows that $X(\DFd)$ is $G$-isomorphic to $G \times^{P} V(\DF)$, where $P$ acts on $V(\DF)$ via the canonical surjection $P \to T$.

\section{The arc space of a horospherical variety of complexity one}\label{s:arc-space}
In this section we describe the arc space of a horospherical variety $X$ of complexity one, in order to compute its stringy $E$-function via motivic integration in Section~\ref{s:E-function}.

Let $K$ be a field extension of $\CC$ and $\O:=\O_K$, $\K:=\K_K$ be the corresponding ring of power series and its fraction field. Denote by $0$ and $\eta$ the closed and generic point of $\Spec \O_K$, respectively. Our goal is to describe the set of $K$-valued points of $\L(X)$ -- that is, the set $X(\O_K)$. To do this we restrict ourselves to the study of a subset $\L_\Gamma(X)$ of $\L(X)$ which has the same motivic measure as $\L(X)$, and we decompose it into \emph{horizontal} and \emph{vertical} arcs, see~\ref{s:hor-ver}, \ref{s:T-non-affine} for the case of $T$-varieties of complexity one. We give a parametrization of the $G(\O)$-orbits in Theorem~\ref{t:arc-horo}. Finally, in~\ref{s:motivic-horo} we show that the corresponding pieces in $\L_{\Gamma}(X)$ are cylinders, hence measurable, and we compute their motivic measure in Theorem~\ref{t:volume}.

\begin{notation}\label{n:setting-arc}
We keep the same notations $M, N, C, G/H, \DF,$ ... as in the previous section for describing a horospherical variety $X = X(\DF)$ of complexity one.  
Especially, for the colored divisorial fan $\DF = \{(\D^i,\F^i)\}_{i \in J}$, and for any $i \in J$, we write
\[
	\D^i = \sum_{y \in C^i} \Delta^i_y \cdot [y],
\]
where $C^i \subset C$ is open dense. The symbol $\sigma_i$ denotes the tail of $\D^i$, and $\Gamma$ is an open dense affine subset of $C$ which does not intersect $\Sp(\DF)$. 
\end{notation}

\subsection{Horizontal and vertical arcs}\label{s:hor-ver}

In this subsection we also assume that $G=T$ is a torus, that $H=\{ e \}$ is trivial, and that $\DF$ consists in a single colored divisor $\D$ with tail $\sigma$ and \emph{affine} locus $C_0$ containing $\Gamma$. The construction following Theorem~\ref{t:classification-horo} shows that $X$ is the $T$-variety $\Spec A(C_0,\D)$.

Let $X_\sigma$ be the toric variety $\Spec \CC[\sigma^\vee \cap M]$ associated with the cone $\sigma$. The product $\Gamma \times X_\sigma$ is a dense open $T$-stable subset of $X$, and so is its subset $\Gamma \times T$. Define
\[
	X_\Gamma = X_\Gamma^K := X(\O) \cap (\Gamma \times T)(\K) := \{ \alpha : \Spec \O_K \to X \mid \alpha(\eta) \in \Gamma \times T \} \subset X(\O),
\]
where $X(\O)$ and $(\Gamma \times T)(\K)$ are both viewed as subsets of $X(\K)$.

\begin{remark}\label{r:generic}
	The subspace 
	\[
		\L_\Gamma(X) := \{ \alpha \in \L(X) \mid \tilde{\alpha}(\eta) \in \Gamma \times T \}
	\]
	has the same motivic measure as $\L(X)$. Indeed, its complement in $\L(X)$ identifies with the arc space of $X_1 := X \setminus (\Gamma \times T)$, which has zero measure, see Remark~\ref{r:measure-closed}. Hence for purposes of motivic integration we may study $\L_\Gamma(X)$ instead of the whole arc space $\L(X)$. Clearly $\L_\Gamma(X)(K) = X_\Gamma^K$.
\end{remark}

Any $K$-valued arc $\alpha$ in $\L_\Gamma(X)(K)$ induces morphisms $\CC[M] \to \K$ and $\CC[\Gamma] \to \K$. Indeed, as $\Gamma \times T$ is an open subset of $X$ we have a commutative diagram
\begin{center}
	\begin{tikzpicture}[description/.style={fill=white,inner sep=2pt}]
		\matrix (m) [matrix of math nodes, row sep=3em,column sep=2.5em, text height=1.5ex, text depth=0.25ex]
			{ \CC[\Gamma] \otimes_\CC \CC[M] &  \K \\
 				\CC[X] & \O. \\};

		\path[-,font=\scriptsize]
			(m-2-1) edge[right hook->] (m-1-1)
				edge[->] node[auto] {$\alpha^*$} (m-2-2)
			(m-1-1) edge[->] node[auto] {$\bar{\alpha}^*$} (m-1-2)
			(m-2-2) edge[right hook->] (m-1-2);
	\end{tikzpicture}
\end{center}
We denote by $\alphaG$ (resp. by $\alphaM$) the restriction of $\bar{\alpha}^*$ to $\CC[\Gamma]$ (resp. to $\CC[M]$). Clearly $\alpha$ is uniquely determined by the data of $\alphaG$ and $\alphaM$. Conversely, if we are given two $\CC$-algebra morphisms $\beta_\Gamma : \CC[\Gamma] \to \K$ and $\beta_M : \CC[M] \to \K$ such that $\beta_\Gamma \otimes \beta_M (\CC[X]) \subset \O$, then there exists an arc $\alpha \in X_\Gamma$ such that $\alphaG = \beta_\Gamma$ and $\alphaM = \beta_M$.

\begin{remark}\label{r:data-arc}
	Following~\cite[Proof of Theorem 4.1]{Ishii:ArcSpaceToric}, if $\alpha \in X_\Gamma$ we parametrize the map $\alphaM$ using the group homomorphisms $\nu_\alpha : M \to \ZZ$ and $\omega_\alpha : M \to \O^*$ defined by the equality $\alphaM(\chi^m) = t^{\langle m,\nu_\alpha \rangle} \omega_\alpha(m)$ for any $m \in M$. We also parametrize $\alphaG$ using the valuation $v_\alpha : \CC[\Gamma] \to \ZZ$ and the function $\qfun_\alpha : \CC[\Gamma] \to \O^*$ such that $\alphaG(f) = t^{v_\alpha(f)} \qfun_\alpha(f)$. Since $\CC(\Gamma)$ is of transcendence degree one over $\CC$, there exists unique $\ell_\alpha \in \ZZ_{\geq 0}$ and $y_\alpha \in \overline{\Gamma}=C$ such that $v_\alpha = \ell_\alpha \ord_{y_\alpha}$.
\end{remark}

The next lemma is a direct consequence of the classification of invariant valuations in \cite[Section 16]{TimashevBook} for toric varieties of complexity one.

\begin{lemma}\label{l:arc2hyperspace}
	Let $w : \CC(X)^* \to \QQ$ be a $T$-invariant discrete valuation. The following statements are equivalent. 
	\begin{enumerate}[(i)]
		\item The restriction $w\mid_{ \CC[X]\setminus \{0\}}$ is non-negative.
		\item There exists $(y,\nu,\ell) \in C(\D)$ such that for any homogeneous element $f \otimes \chi^m$ in $\CC[X]$, we have 
			\[
				w(f \otimes \chi^m) = \ell \ord_y(f) + \langle m,\nu \rangle.
			\]
	\end{enumerate}
	Hence if $\alpha \in X_\Gamma$ is an arc, then the triple $(y_\alpha,\nu_\alpha,\ell_\alpha)$ defined above is a point of the hypercone $C(\D) \cap \N$. 
\end{lemma}

The following disjoint subsets of $X_\Gamma$ are defined according to the position of the image of $0 \in \Spec \O$.
\begin{definition}
	Define the \emph{horizontal part} of $X_\Gamma$ by
	\[
		\hor := \{ \alpha : \Spec \O \to X \mid \alpha(0) \in \Gamma \times X_\sigma \text{ and } \alpha(\eta) \in \Gamma \times T \}
	\]
	and its \emph{vertical part} by $\ver:=X_\Gamma \setminus \hor$, that is
	\[
		\ver = \{ \alpha : \Spec \O \to X \mid \alpha(0) \not\in \Gamma \times 	X_\sigma \text{ and } \alpha(\eta) \in \Gamma \times T \}.
	\]
	The condition $\alpha(0) \not\in \Gamma \times 	X_\sigma$ means that the closed point $\alpha(0)$ belongs to the fiber over a point of $C_0 \setminus \Gamma$ of the quotient map $X \to C_0$.
\end{definition}

We may characterize elements $\alpha \in \ver$ in terms of their associated maps $\alphaG$.
\begin{lemma}\label{l:arc-injective}
	Let $\alpha \in X_\Gamma$ be an arc. The following properties are equivalent:
	\begin{enumerate}
		\item $\alpha \in \ver$,
		\item $\alphaG(\CC[\Gamma]) \not \subset \O$,
		\item $y_\alpha \in C_0 \setminus \Gamma$ and $\ell_\alpha \neq 0$,
	\end{enumerate}
	where $y_\alpha$ and $\ell_\alpha$ are as in Remark~\ref{r:data-arc}. In particular, if one of the above conditions is satisfied, then $\alphaG$ is injective.
\end{lemma}

\begin{proof}
	The equivalence between (2) and (3) is clear. Let us now prove that (1) and (3) are equivalent.
	
	If $y_\alpha \in \Gamma$, then $\nu_\alpha \in \sigma \cap N$ by definition of $C(\D)$. Moreover $\ord_{y_\alpha}$ is non-negative on $\CC[\Gamma]$, the image of 
	$\CC[\Gamma] \otimes_\CC \CC[X_{\sigma}]$ by $\alphaG \otimes \alphaM$ is contained in $\O$, hence $\alpha \in \hor$. So $\alpha \in \ver$ implies $y_\alpha \in C_0 \setminus \Gamma$. Conversely if $y_\alpha \in C_0 \setminus \Gamma$ and $\ell_\alpha \neq 0$, then there exists $f \in \CC[\Gamma]$ such that $\ell_\alpha \ord_{y_\alpha}(f) < 0$, so that $\alphaG(\CC[\Gamma]) \not \subset \O$. In particular $\alphaG \otimes \alphaM(\CC[\Gamma] \otimes_\CC \CC[X_\sigma]) \not \subset \O$, and $\alpha \in \ver$.
	
	Now assuming $\alphaG$ is not injective, its kernel is the ideal of a $\CC$-point, so $\ell_\alpha=0$ and $\nu_\alpha \in \sigma$. This implies that $\alphaG \otimes \alphaM(\CC[\Gamma] \otimes_\CC \CC[X_{\sigma}]) \subset \O$, which means that $\alpha \in \hor$.
\end{proof}

Let us now study the $T(\O)$-action on the space $X_\Gamma$. Consider
\[
	\val : X_\Gamma \to C(\D) \cap \N,
\]
which to an arc $\alpha$ associates the point $(y_\alpha,\nu_\alpha,\ell_\alpha)$ from Lemma~\ref{l:arc2hyperspace}. Our goal is to decompose into $T(\O)$-orbits the fibers of $\val$ above points of $\hor$ and $\ver$. In the case of $\hor$ the result is immediate from \cite[Theorem 4.1]{Ishii:ArcSpaceToric}.

\begin{lemma}\label{l:decomposition-hor}
	We have a decomposition
	\[
		\hor 	= \Gamma(\O) \times (X_\sigma(\O) \cap T(\K)) = \Gamma(\O) \times \bigsqcup_{\nu \in \sigma \cap N} \C_\nu,
	\]
	where $\alpha \in \hor$ is identified  with the pair $(\alphaG,\alphaM) \in \Gamma(\O) \times (X_\sigma(\O) \cap T(\K))$, and
	\[
		\C_\nu := \{ \alphaM \mid \alpha \in \hor, \nu_\alpha = \nu \}
	\]
	is a $T(\O)$-orbit of $X_\sigma(\O) \cap T(\K)$. Moreover
	\[
		\val(\hor) = \{ (C,\nu,0) \mid \nu \in \sigma \cap N \}
	\]
	and $\val^{-1}(C,\nu,0) = \Gamma(\O) \times \C_\nu$ for any $\nu \in \sigma \cap N$.
\end{lemma}

We now state a similar result for the subset $\ver$.

\begin{lemma}\label{l:decomposition-ver}
	Let $\alpha$ be an arc in $\ver$ and $\pi \in \CC(C)^*$ be a uniformizer of $y_\alpha$. There exists a one-to-one correspondence between
		\[
			\verY := \val^{-1}(\val(\alpha)) = \{ \beta \in \ver \mid y_\beta = y_\alpha, \nu_\beta = \nu_\alpha, l_\beta =  l_\alpha\}
		\]
		and the set of pairs $(\omega,u)$, where $\omega : M \to \O^*$ is a group homomorphism and $u \in \O^*$. 
		The correspondence is given by 	$\beta \mapsto (\omega_\beta,\qfun_\beta(\pi))$, where $\qfun_\beta$ is the function from Remark~\ref{r:data-arc}. It identifies the fiber $\verY$ with the product $C_{\nu_{\alpha}} \times \O^*$. Furthermore, we have
	\begin{equation}\label{e:val-ver}
		\val (\ver) = \{ (y,\nu,\ell) \in C(\D) \cap \N \mid y \in C_0 \setminus \Gamma \text{ and } \ell \geq 1 \}.
	\end{equation}
\end{lemma}

\begin{proof}
	We first prove the surjectivity of the correspondence. Fix a pair $(\omega,u) \in T(\O) \times \O^*$ and write $y = y_\alpha$, $\nu = \nu_\alpha$, and $\ell = \ell_\alpha$. From \cite[Theorem 7.16]{Eisenbud:CommAlg}, we know that there exists a unique morphism $\phi : \hat \OO_y = \CC[[\pi]] \to \O$ such that $\phi(\pi)=t^{\ell} u$, where $\hat \OO_y$ is the formal completion of $(\OO_y, \pi \OO_y)$. Restricting to $\CC[C_0]$, since $\ell \neq 0$, $\phi$ induces an injective homomorphism $\CC[C_0] \to \O$, see Lemma~\ref{l:arc-injective}. Extending it to the fraction field we obtain a homomorphism $\CC(C_0) \to \K$, whose restriction to $\CC[\Gamma]$ is denoted by $\lambda$. We define an arc $\beta \in \ver$ by setting $\bar \beta^*(f \otimes \chi^m) = \lambda(f) t^{\langle m,\nu \rangle} \omega(m)$ for any homogeneous element $f \otimes \chi^m \in \CC[\Gamma] \otimes_\CC \CC[\sigma^\vee \cap M]$. The arc $\beta$ is in $\ver$ since $\ell_\beta = \ell \neq 0$ and $y_\beta = y \in C_0 \setminus \Gamma$.
	
	To prove that the correspondence is injective, it is enough to check that for $\beta \in \ver$, the map $\lambda : \CC[\Gamma] \to \K$ above is uniquely determined by $\ell=\ell_\beta$, $y=y_\beta$, and $\qfun_\beta(\pi)$, where $\pi \in \CC(C)^*$ is a uniformizer of $y$. Clearly $\lambda$ induces a continuous morphism $\bar \lambda : (\OO_y,\pi \OO_y) \to (\O,t \O)$, i.e., a morphism such that $\bar \lambda (\pi \OO_y) \subset t \O$. Taking formal completions we also get a morphism $\tilde{\lambda} : \hat \OO_y = \CC[[\pi]] \to \O$, see \cite[\S (23.H)]{Matsumura:CommAlg}. Using again \cite[Theorem 7.16]{Eisenbud:CommAlg} we see that $\tilde{\lambda}$ is the unique morphism sending $\pi$ to $t^\ell \qfun_\beta(\pi)$. We conclude by noticing that $\lambda$ is uniquely determined by $\tilde{\lambda}$, which proves the injectivity.
	
	To conclude the proof we check Equation~\eqref{e:val-ver}. The direct inclusion of $\val(\ver)$ has been proved by combining Lemmas~\ref{l:arc2hyperspace} and~\ref{l:arc-injective}. For the other inclusion, consider $(y,\nu,\ell) \in C(\D) \cap \N$ with $y \in C_0 \setminus \Gamma$ and $\ell \neq 0$. Choose a uniformizer $\pi$ of $y$ and consider the unique morphism $\hat \OO_y = \CC[[\pi]] \to \O$ which sends $\pi$ to $t^\ell$. By the same argument as before it implies the existence of a morphism $\bar \beta^* : A(C_0,\D) \to \O$ such that $\bar \beta^*(f \otimes \chi^m) := t^{\ell \ord_y(f) + \langle m,\nu \rangle}$, hence the existence of an arc $\beta \in \ver$ such that $\val(\beta) = (y,\nu,\ell)$.
	\end{proof}

\subsection{The arc space of a normal $T$-variety of complexity one}\label{s:T-non-affine}

In this section we generalize the results of \ref{s:hor-ver} to the non-affine case. We still assume that $G=T$ is a torus and $H=\{ e \}$ is trivial.

We want to describe the arc space of $X=X(\DF)$ in terms of the arc spaces of the affine charts $X(\D)$ for polyhedral divisors $\D \in \DF$, which may have non-affine loci. So to apply the previous results we need to replace $X$ with a $T$-variety $\tilde X$ associated with a divisorial fan $\tilde \DF$ such that all the polyhedral divisors in $\tilde \DF$ are defined on open dense \emph{affine} subsets of $C$. We call this process \emph{affinization}:

\begin{notation}\label{n:affine-fan}
Keeping the same notations as in \ref{s:hor-ver}, for any $\D^i\in \DF$, we consider a finite open dense affine covering $(C^i_j)_{j \in J_i}$ of the curve $C^i$. Denote by $\D_j^i$ the polyhedral divisor obtained as the restriction of $\D^i$ to the open affine subset $C^i_j$, and by $\tilde{\DF}$ the divisorial fan constituted of all the $\D^i_j$. Then the $T$-variety $\tilde X = X(\tilde \DF)$ does not depend on the choices of the affine coverings, and its support $|\tilde \DF|$ coincides with that of $\DF$.
\end{notation}

From \cite[Theorem 3.1]{AltmannHausen} it follows that the inclusions $A(C^i,\D^i) \hookrightarrow A(C^i_j,\D_j^i)$ induce a proper $T$-equivariant birational morphism 
\begin{equation}\label{e:Etilde}
	q : X(\tilde \DF) \to X(\DF).
\end{equation}

As in Section~\ref{s:hor-ver}, we introduce a morphism $\val$ defined on $X_\Gamma$, whose image will be contained in
\begin{equation}\label{e:HGamma}
	|\DF|_\Gamma := \bigcup_{\D \in \DF} C_\Gamma(\D) \subset \N_\QQ,
\end{equation}
where $C_\Gamma(\D) = C(\D) \setminus \left\{ (y,\nu,\ell) \mid y \in \Gamma, \nu \in N_\QQ, \ell > 0 \right\}$. Clearly $|\DF|_\Gamma = |\tilde \DF|_\Gamma$.

The next result extends Lemmas~\ref{l:decomposition-hor} and~\ref{l:decomposition-ver} to the $T$-variety $X$.

\begin{proposition}\label{p:T-varieties}
	Let $\DF$ be a divisorial fan on $(C,T)$ and write $X=X(\DF)$. Consider a dense open subset $\Gamma\subset C$ which does not contain any special point. There exists a surjective map
	\begin{equation}\label{e:val}
		 \begin{array}{clcl}
		 	\val : & X_\Gamma & \to & |\DF|_\Gamma \cap \N \\
		 	& \alpha & \mapsto & (y_\alpha,\nu_\alpha,\ell_\alpha),
		 \end{array}
	\end{equation}
where $|\DF|_\Gamma$ is as in Equation~\eqref{e:HGamma}, and $\N$ as in Definition~\ref{d:hyperspace}. Moreover
	\begin{equation*}
		\val^{-1}(y_\alpha,\nu_\alpha,\ell_\alpha) = \begin{cases}
		 \Gamma(\O) \times \C_{\nu_{\alpha}} & \text{if $\ell_\alpha = 0,$} \\
		\O^* \times \C_{\nu_\alpha} & \text{if $\ell_\alpha \geq 1$.}
		\end{cases}
	\end{equation*}
\end{proposition}

\begin{proof}
	Recall the morphism $q : \tilde X \to X$ from Equation~\eqref{e:Etilde}. We first check that $X_\Gamma$ identifies with
	\[
		\tilde X_\Gamma := \{ \alpha \in \tilde X(\O) \mid \alpha(\eta) \in \Gamma \times T \}.
	\] 
	Indeed, we know that $q\mid_{\Gamma \times T}$ is the identity map. Write $X_1 := X \setminus (\Gamma \times T)$. By the valuative criterion of properness, the morphism $q$ induces a bijection
	\begin{center}
		\begin{tikzpicture}[description/.style={fill=white,inner sep=2pt}]
			\matrix (m) [matrix of math nodes, row sep=3em,column sep=2.5em, text height=1.5ex, text depth=0.25ex]
			{ 	\L(\tilde X) \setminus \L(q^{-1}(X_1)) &  \L(X) \setminus \L(X_1). \\};

			\path[-,font=\scriptsize]
			(m-1-1) edge[->] node[auto] {$\sim$} (m-1-2);
		\end{tikzpicture}
	\end{center}
	This implies that $X_\Gamma \cong \tilde X_\Gamma$, so we will assume in the rest of the proof that $\DF=\tilde \DF$.
	
	We now prove that $\val$ is well-defined and surjective. Clearly the arc space has a covering
	\[
		X_\Gamma = \bigcup_{\D \in \DF} X(\D)_\Gamma,
	\]
	where $X(\D)_\Gamma := \{ \alpha \in X_\Gamma \mid \alpha(0) \in X(\D) \}$. Fixing a uniformizer for each point of $C\setminus \Gamma$ and combining Lemmas~\ref{l:decomposition-hor} and~\ref{l:decomposition-ver} we get well-defined surjective maps
	\[
		\val_{\D} : X(\D)_\Gamma \to C_\Gamma(\D) \cap \N,
	\]
	which are clearly compatible and glue into $\val$. Finally, the description of the fibers of $\val$ is a consequence of Lemma~\ref{l:decomposition-hor} when $\ell_\alpha=0$ and of Lemma~\ref{l:decomposition-ver} when $\ell_\alpha \geq 1$.
\end{proof}

\subsection{The general case}\label{s:arcs-general}

Let us now treat the case of a horospherical $G$-variety $X$ of complexity one. We start by recalling the complexity zero case.

Let $G/H$ be a horospherical homogeneous space and $P$, $I$, $T$, $M$ and $N$ be as in Notation~\ref{n:setting-arc}. A horospherical embedding $Y$ of $G/H$ is parametrized by a colored fan $\Sigma$ on $N_\QQ$, see~\cite[Theorem 3.3]{Knop:Luna-Vust}. Let $|\Sigma|$ be the support of $\Sigma$, that is, the reunion $|\Sigma| := \bigcup_{(\sigma,\F) \in \Sigma} \sigma$, where $\sigma$ is a strongly convex cone and $\F \subset \F_{G/H}$ is a subset of colors. The next theorem describes the arc space of $Y$ in terms of $|\Sigma|$.  

\begin{theorem}[{\cite[Theorem 3.1]{Batyrev-Moreau}}]\label{t:Batyrev-Moreau}
	Let $Y$ be a horospherical $G/H$-embedding defined by a colored fan $\Sigma$. Then there exists a surjective map
	\[
		\mathcal{V} : Y(\O) \cap (G/H)(\K) \to |\Sigma| \cap N
	\]
whose fiber over $\nu \in |\Sigma| \cap N$ is a $G(\O)$-orbit $\Omega_\nu$. 
In particular we obtain a one-to-one correspondence between the lattice points in $|\Sigma| \cap N$ and the $G(\O)$-orbits in $Y(\O) \cap (G/H)(\K)$.
\end{theorem}

Let us sketch the construction of the map $\mathcal{V}$. After discoloring, using the valuative criterion of properness we may assume that $Y=G \times^{P} Y_{\Sigma}$, where $Y_\Sigma$ is the $T$-variety with fan $\Sigma$. Since $T=P/H \subset G/H$, the closure $Y_\Sigma = \overline T$ embeds into $Y$. We have a commutative diagram:
\begin{equation}\label{e:diagram-arc}
	\begin{tikzpicture}[description/.style={fill=white,inner sep=2pt},baseline=(current bounding  box.center)]
			\matrix (m) [matrix of math nodes, row sep=3em,column sep=2.5em, text height=1.5ex, text depth=0.25ex]
		{ 	Y_\Sigma(\O) \cap T(\K) &  {|\Sigma|} \cap N \\
			Y(\O) \cap (G/H)(\K) &  \\
		};
		\path[-]
			(m-1-1) edge[->] node[auto] {$\Psi$} (m-1-2)
			(m-1-1) edge[right hook->] (m-2-1)
			(m-2-1) edge[->] node[below] {$\mathcal{V}$} (m-1-2);
	\end{tikzpicture}
\end{equation}
where the map $\Psi$ comes from \cite[Corollary 4.3]{Ishii:ArcSpaceToric}. Denote by $\C_\nu$ the fiber over $\nu \in |\Sigma| \cap N$ of $\Psi$. The map $\mathcal{V}$ extends $\Psi$ in such a way that $\Omega_\nu$ is the unique $G(\O)$-orbit containing $\C_\nu$. 

For the following theorem, which deals with the complexity one case, we use the setting of Notation~\ref{n:setting-arc}. 

\begin{theorem}\label{t:arc-horo}
	Let $X:=X(\DF)$, where $\DF$ is a colored divisorial fan on $(C,G/H)$. Let  $\Gamma$ be a dense open affine subset of $C$ which does not contain any special point. There exists a surjective map
	\[
		\val : \begin{array}{lcl}
		         X_\Gamma  = X(\O)\cap (\Gamma\times G/H)(\K) & \to & |\DF|_\Gamma \cap \N \\
		 	 \alpha & \mapsto & (y_\alpha,\nu_\alpha,\ell_\alpha)
		\end{array} \:\: \text{where} \:\:\:
		|\DF|_\Gamma := \bigcup_{(\D,\F) \in \DF} C_\Gamma(\D) \subset \N_\QQ
	\]
	and $\N, \N_\QQ$ are as in Definition~\ref{d:hyperspace}.
	Moreover it satisfies
	\begin{equation*}
		\val^{-1}(y_\alpha,\nu_\alpha,\ell_\alpha) = \begin{cases}
		 \Gamma(\O) \times \Omega_{\nu_{\alpha}} & \text{if $\ell_\alpha = 0,$} \\
		\O^* \times \Omega_{\nu_\alpha} & \text{if $\ell_\alpha \geq 1$,}
		\end{cases}
	\end{equation*}
	where the $\Omega_{\nu_\alpha}$ are as in Theorem~\ref{t:Batyrev-Moreau}.
\end{theorem}

\begin{proof}
	First we discolor $X$ using the map $\pi_{dis} : X_{dis} \to X$ from Equation~\eqref{e:discolor}. We have $X_{dis} = G \times^P V(\DF)$, where $V(\DF)$ is as in Equation~\eqref{e:V(E)}. Then we replace the colored divisorial fan $\DF$ with another fan $\tilde \DF$ using the affinization map $q$ from Equation~\eqref{e:Etilde}, obtaining a proper birational morphism $\Theta : G \times^P V(\tilde \DF) \to X$ which restricts to the identity on $\Gamma \times G/H$. By the valuative criterion of properness, it is equivalent to consider arcs on $X$ or on $G \times^P V(\tilde \DF)$, so from now on we assume $X=G \times^P V(\tilde \DF)$.
	
	Next we compare $X_\Gamma$, the set of $K$-valued formal arcs on $X$ with generic point in $\Gamma \times T$, with
	\[
		V(\tilde \DF)_\Gamma = \{ \alpha : \Spec \O \to V(\tilde \DF) \mid \alpha(\eta) \in \Gamma \times T \} \subset \L(V(\tilde \DF))(K).
	\]
	By construction of $X=G \times^P V(\tilde \DF)$, the closure of $\Gamma \times T$ embeds in $X$ and identifies with $V(\tilde \DF)$. The variety $V(\tilde \DF)$ will play the same role as $Y_\Sigma$ in Equation~\eqref{e:diagram-arc}. We have a map $\Psi : V(\tilde \DF)_\Gamma \to |\tilde \DF|_\Gamma \cap \N$, which is the map $\val$ from Proposition~\ref{p:T-varieties}.
	
	From the proof of \cite[Theorem 3.1]{Batyrev-Moreau} we obtain a map $\mathcal{V} : (G/H)(\K) \to N$ which is constant on the $G(\O)$-orbits, and whose restriction to $T(\K) \subset (G/H)(\K)$ is constructed from the standard valuation map $\ord : \K^* \to \ZZ$ as in \cite[Lemma 3.2]{Batyrev-Moreau}. We use it to build part of a map from $(\Gamma \times G/H)(\K)$ to $\N$, which we will then restrict to $X_\Gamma$. 
	
	To get the other part of the map, as in Remark~\ref{r:data-arc} we parametrize $\alpha \in \Gamma(\K)$ by the pair $(y_\alpha,\ell_\alpha)$, where $\alpha^*(f) = t^{\ell_\alpha \ord_{y_\alpha}(f)} \qfun_\alpha(f)$ for any $f \in \CC[\Gamma]$. Finally
	\[
		\begin{array}{clcl}
			\val : &(\Gamma\times G/H)(\K)  & \to & \N \\
			& (\alpha,\beta) & \mapsto & (y_\alpha,\nu_\beta,\ell_\alpha),
		\end{array}
	\]
	where $\nu_\beta = \mathcal{V}(\beta)$. Restricted to $V(\tilde \DF)_\Gamma$ this map is simply $\Psi$, and it is constant on the $G(\O)$-orbits.
	
	We now study the restriction of $\val$ to $X_\Gamma$, which we denote again in the same way. Consider the quotient map $\varphi : X \to G/P$ and the corresponding map $X(\O) \to (G/P)(\O)$ on the set of $\O$-valued points. We denote by $\vinf : X_\Gamma \to (G/P)(\O)$ its restriction to $X_\Gamma$. It extends to the natural map
	\[
		\bvinf : \Gamma(\K) \times (G/H)(\K) \to (G/P)(\K).
	\] 
	Since $(G/P)(\K)=(G/P)(\O)$ by the valuative criterion of properness, the image of $\bvinf$ is in fact contained in $(G/P)(\O)$, which is equal to $G(\O)/P(\O)$ by the local triviality of the quotient map $G \to G/P$. Moreover the map $\vinf$ is $G(\O)$-equivariant. If $\alpha \in X_\Gamma$, then by transitivity and $G(\O)$-equivariance there exists $g \in G(\O)$ such that $\vinf(g \cdot \alpha) =p_0$, where $p_0 = P(\O) \in (G/P)(\O)$. Hence
	\[
		g \cdot \alpha \in \vinf^{-1}(p_0) \subset \bvinf^{-1}(p_0) = (\Gamma \times T)(\K).
	\]
	It follows that $g \cdot \alpha \in (\Gamma \times T)(\K) \cap X_\Gamma$, and $(\Gamma \times T)(\K) \cap X_\Gamma \subset V(\tilde \DF)_\Gamma$. Indeed if $\beta \in (\Gamma \times T)(\K) \cap X_\Gamma$, then $\beta(0) \in \overline{\beta(\eta)} \subset \overline{\Gamma\times T} = V(\tilde \DF)$, thus $\beta$ is an element of $V(\tilde \DF)_\Gamma$. As a result $g \cdot \alpha \in V(\tilde \DF)_\Gamma$, and $\val(\alpha) = \val(g \cdot \alpha) = \Psi(g \cdot \alpha)$. This implies that $\val(X_\Gamma)=\Psi(V(\tilde \DF)_\Gamma) = |\tilde \DF|_\Gamma \cap \N$.
	
	To conclude we compute the fibers of the restriction $\val : X_\Gamma \to |\tilde \DF|_\Gamma \cap \N$. 
	Consider $(y,\nu,\ell) \in |\tilde \DF|_\Gamma \cap \N$. For any $\alpha \in \val^{-1}(y,\nu,\ell)$, by a previous argument there exists $g \in G(\O)$ such that $g \cdot \alpha \in V(\tilde \DF)_\Gamma \cap \val^{-1}(y,\nu,\ell)$. Finally, by Proposition~\ref{p:T-varieties}, we know that 
	\begin{equation*}
		V(\tilde \DF)_\Gamma \cap \val^{-1}(y,\nu,\ell) = \begin{cases}
														\Gamma(\O) \times \C_\nu & \text{if $\ell=0$,} \\
														\O^* \times \C_\nu & \text{if $\ell \geq 1$.} 
													\end{cases} \qedhere
	\end{equation*}
\end{proof}

\subsection{Motivic volumes}\label{s:motivic-horo} Assuming that $X$ is smooth, we now compute the motivic measure of the fibers of the map $\val$ from Theorem~\ref{t:arc-horo}. 

We start by studying truncations of arcs in $X_\Gamma$. Using the discoloration morphism, we may assume that $\DF$ has trivial coloration, so that $X= G \times^P V(\DF)$. Recall that there is a surjective quotient map $\varphi : X \to G/P$. The next result, obtained from Theorem~\ref{t:arc-horo}, gives a comparison between the jet spaces of $X$ and those of the flag variety $G/P$. 

\begin{lemma}\label{l:surj-arc}
	Consider $\xi := (y,\nu,\ell) \in |\DF|_\Gamma \cap \N$ and $q \in \NN$, where $|\DF|_\Gamma$ is as in Equation~\eqref{e:HGamma}. Then the restriction to $\tronc_q(\val^{-1}(\xi))$ of the bundle of $q$-jets $\varphi_q : \L_q(X) \to \L_q(G/P)$ is a bundle with fiber isomorphic to $\tronc_q'\left(V(\DF)_\Gamma \cap \val^{-1}(\xi)\right)$, where $\tronc_q' : \L(V(\DF)) \to \L_q(V(\DF))$ is the truncation map.
\end{lemma}

\begin{proof}
	We have a commutative diagram
	\begin{center}
	\begin{tikzpicture}[description/.style={fill=white,inner sep=2pt}]
			\matrix (m) [matrix of math nodes, row sep=3em,column sep=2.5em, text height=1.5ex, text depth=0.25ex]
		{ 	\val^{-1}(\xi) &  \L(G/P) \\
			\tronc_q(\val^{-1}(\xi)) & \L_q(G/P) \\
		};
		\path[-]
			(m-1-1) edge[->] (m-1-2)
			(m-1-1) edge[->] (m-2-1)
			(m-1-2) edge[->] (m-2-2)
			(m-2-1) edge[->] (m-2-2);
	\end{tikzpicture}
	\end{center}
	where the vertical maps are arc truncations and the horizontal maps are induced by $\varphi : X \to G/P$. The vertical map on the left-hand side is clearly surjective, and the other also is since $G/P$ is smooth (see \cite{Greenberg1966:henselian}). Finally, the top horizontal map is surjective by the proof of Theorem~\ref{t:arc-horo}, hence so is the bottom map. The last claim in the lemma follows from the description of the fibers of the top map.
\end{proof}

Using Lemma~\ref{l:surj-arc} we now show that the fibers of $\val$ are cylinders, hence measurable. Lemma~\ref{l:vol-hor} deals with the horizontal fibers, while Lemma~\ref{l:vol-ver} takes care of the vertical fibers. 

\begin{lemma}\label{l:vol-hor} Let $X$ be as in Notation~\ref{n:setting-arc}.
	Assume that $X$ is smooth of dimension $d$ and that $\DF$ has trivial coloration. Consider $\xi := (y,\nu,0) \in |\DF|_\Gamma \cap \N$.
	Let $n$ be the rank of the lattice $N$, $\sigma\in \Sigma(\DF)$ be a cone containing $\nu$ and $r := \dim \sigma$. The linear part $\sigma^\vee \cap (-\sigma^\vee)$ of $\sigma^\vee$ is a $(d-1-r)$-dimensional vector space $V_\QQ$. Let $W_\QQ$ be a complement of $V_\QQ$ in $M_\QQ$ and $u_1,\dots,u_r$ be an integral basis of $\sigma^\vee \cap W_\QQ$.
	
	Then for any $q \geq \max (\{ \langle u_j,\nu \rangle \mid 1 \leq j \leq r \})$, the set $\val^{-1}(\xi) \cap V(\DF)_\Gamma$ is a cylinder with $q$-basis 
	\[
		\L_q(\Gamma) \times (\AA^1 \setminus 0)^{n} \times \AA^{qn-\sum_{j=1}^r \langle u_j,\nu \rangle}.
	\]
\end{lemma}

\begin{proof}
	By the smoothness criterion of \cite[Chap. II]{KKMS73}, any cone $\sigma \in \Sigma(\DF)$ is generated by part of a basis of $N$. Consider $\alpha \in \val^{-1}(\xi) \cap V(\DF)_\Gamma$ and  the morphisms $\alpha^*_\Gamma : \CC[\Gamma] \to \O$ and $\alpha^*_M : \CC[\sigma^\vee \cap M] \to \O$ from Remark~\ref{r:data-arc}.
	The maps $\alpha^*_\Gamma$ and $\alpha^*_M$ induce arcs on $\Gamma$ and $X_\sigma$, which can be truncated to $q$-jets $\alpha_\Gamma' \in \L_q(\Gamma)$ and $\alpha_M' \in \L_q(X_\sigma)$. Moreover the set $\{ \alpha_\Gamma' \mid \alpha \in \val^{-1}(\xi) \cap V(\DF)_\Gamma \}$ is isomorphic to $\L_q(\Gamma)(K)$ since $\Gamma$ is smooth. To conclude we use an adapted version of \cite[Lemma 3.4]{Batyrev-Moreau} where we do not assume the cone to be full-dimensional. The result of the lemma implies that
	\[
		\{ \alpha_M' \mid \alpha \in \val^{-1}(\xi) \cap V(\DF)_\Gamma \} \simeq (\AA^1(K) \setminus 0)^{n} \times \AA^{qn-\sum_{j=1}^r \langle u_j,\nu \rangle}(K). \qedhere
	\]
\end{proof}

\begin{lemma}\label{l:vol-ver} Let $X$ be as in Notation~\ref{n:setting-arc}. Assume that $X$ is smooth of dimension $d$, that $\DF$ has trivial coloration, and denote by $n$ the rank of the lattice $N$. Consider $\xi := (y,\nu,\ell) \in |\DF|_\Gamma \cap \N$ such that $\ell \geq 1$. 
	Let $(\D,\F) \in \DF$ be such that $\xi \in C_\Gamma(\D)$ and let $\sigma$ be the tail of $\D$. 
	Denote by $r$ the dimension of $\sigma$ and by $s \geq r+1$ that of the Cayley cone $C_y(\D)$. The linear part $C_y(\D)^\vee \cap (-C_y(\D)^\vee)$ of $C_y(\D)^\vee$ is a $(d-s)$-dimensional vector space $V_\QQ$. Let $W_\QQ$ be a complement of $V_\QQ$ in $\Lambda_\QQ = M_\QQ \oplus \QQ$ and denote by $\{ (u_1,0),\dots,(u_{r},0), w_{1},\dots, w_{s-r} \}$
	an integral basis of $C_y(\D)^\vee \cap W_\QQ$, where the $u_i$ are in $\sigma^\vee \cap M$.
	
	Then for any $q$ greater than $\max \left(\{ \langle u_j,\nu \rangle \mid 1 \leq j \leq r \} \cup \{ \langle w_{j},(\nu,\ell) \rangle \mid 1 \leq j \leq s-r \}\right)$, the set $\val^{-1}(\xi) \cap V(\DF)_\Gamma$ is a cylinder with $q$-basis
	\[
		(\AA^1 \setminus 0)^{n+1} \times \AA^{q(n+1)-(\sum_{j=1}^r \langle u_j,\nu \rangle + \sum_{j=1}^{s-r} \langle w_{j},(\nu,\ell) \rangle)}.
	\]
\end{lemma}

\begin{proof}
	Let $\pi \in \CC(C)^*$ be a uniformizer of $y$. Writing $\Lambda = M \oplus \ZZ$ we have 
	\[
		\CC[C_y(\D)^\vee \cap \Lambda] = \{ \pi^k \otimes \chi^m \mid (m,k) \in C_y(\D)^\vee \cap \Lambda \}.
	\]
	Denote by $X_{C_y(\D)} := \Spec \CC[C_y(\D)^\vee \cap \Lambda]$ the associated toric variety. The restricted fiber $\val^{-1}(\xi) \cap V(\DF)_\Gamma$ identifies with the $(T \times \mathbb{G}_m)(\O)$-orbit $C_{\nu,\ell}$ of $X_{C_y(\D)}(\O) \cap (T \times \mathbb{G}_m)(\K)$ given in \cite[Theorem 4.1]{Ishii:ArcSpaceToric}.
	
	Indeed, as in the proof of Lemma~\ref{l:decomposition-ver}, any arc $\alpha$ in the restricted fiber induces a morphism $\CC[[\pi]] \otimes_\CC \CC[M] \to \K$, which restricts to $\CC[C_y(\D)^\vee \cap \Lambda] \to \O$. Conversely, if we have an arc $\beta \in \L(X_{C_y(\D)})(K)$ with the property above, it induces a morphism $\CC[\pi,\pi^{-1}] \otimes_\CC \CC[M] \to \K$. Since the completion of $(\CC[\pi,\pi^{-1}],\pi \CC[\pi,\pi^{-1}])$ is $\CC((\pi))$, by \cite[Theorem 7.16]{Eisenbud:CommAlg} $\beta$ gives a morphism $\CC[[\pi]] \otimes_\CC \CC[M] \to \K$ which restricts to a co-morphism of an arc in $\val^{-1}(\xi) \cap V(\DF)_\Gamma$. The rest of the proof follows from the suitably modified version of \cite[Lemma 3.4]{Batyrev-Moreau} already used in Lemma~\ref{l:vol-hor}.
\end{proof}

The next result gives the motivic volume of the fibers of the surjective map $\val : X_\Gamma \to |\DF|_\Gamma \cap \N$. Thus we obtain, up to a subset of motivic measure zero, a decomposition of the arc space of $X$ as a disjoint union of measurable subsets of known motivic volume, which concludes our study of $\L(X)$. 

\begin{theorem}\label{t:volume}
	Let $\DF$ be a colored divisorial fan on $(C,G/H)$ such that $X=X(\DF)$ is smooth of dimension $d$. Without loss of generality, we may assume that $\DF$ has trivial coloration.
	
	Consider $\xi := (y,\nu,\ell) \in |\DF|_\Gamma \cap \N$, where $|\DF|_\Gamma := \bigcup_{(\D,\F) \in \DF} C_\Gamma(\D) \subset \N_\QQ$ and $\N, \N_\QQ$ are as in Definition~\ref{d:hyperspace}. Let $(\D,\F) \in \DF$ be such that $\xi \in C_\Gamma(\D)$ and let $\sigma$ be the tail of $\D$. Denote by $r$ the dimension of $\sigma$ and by $s \geq r+1$ be that of the Cayley cone $C_y(\D)$. The linear part $C_y(\D)^\vee \cap (-C_y(\D)^\vee)$ of $C_y(\D)^\vee$ is a $(d-s)$-dimensional vector space $V_\QQ$. Let $W_\QQ$ be a complement of $V_\QQ$ in $\Lambda_\QQ = M_\QQ \oplus \QQ$ and denote by $\{ (u_1,0),\dots,(u_{r},0),w_{1},\dots, w_{s-r}\}$ an integral basis of $C_y(\D)^\vee \cap W_\QQ$, where the $u_i$ are in $\sigma^\vee \cap M$. Then 
	\begin{equation*}
		\mu_X(\val^{-1}(\xi)) = \begin{cases}
		 [\Gamma] [G/H] \LL^{-\sum_{j=1}^r \langle u_j,\nu \rangle} & \text{if $\ell = 0,$} \\
		 [G/H] (\LL-1) \LL^{-\sum_{j=1}^s \langle u_j,\nu \rangle - \sum_{j=1}^{s-r} \langle w_{j}, (\nu, \ell)\rangle} & \text{if $\ell \geq 1$.}
														\end{cases}
	\end{equation*}
\end{theorem}

\begin{proof}
	For $\ell=0$, since $G/P$ is smooth, by Lemmas~\ref{l:surj-arc} and~\ref{l:vol-hor} it follows that for large enough $q\geq 0$,
	\begin{align*}
		\mu_X(\val^{-1}(C,\nu,0)) &= [\tronc_q'\left(V(\DF)_\Gamma \cap \val^{-1}(\xi)\right)][\L_{q}(G/P)]\LL^{-qd} \\
		&= [\L_{q}(\Gamma)](\LL - 1)^{n}[G/P]\LL^{q(n+\dim(G/P) -d) - \sum_{j=1}^r \langle u_j,\nu \rangle}
	\end{align*}
	The result for $\ell=0$ follows from the equalities $\dim(G/P) = d-1-n$, $[G/H] = (\LL - 1)^{n}[G/P]$, and $\L_{q}(\Gamma)= [\Gamma]\LL^{q}$. For the case $\ell\geq 1$, we have by Lemmas ~\ref{l:surj-arc} and ~\ref{l:vol-ver} the equality
$$\mu_X(\val^{-1}(\xi)) = (\LL - 1)^{n+1}[G/P]\LL^{q(n+1+\dim(G/P)-d) -\sum_{j=1}^s \langle u_j,\nu \rangle - \sum_{j=1}^{s-r} \langle w_{j}, (\nu, \ell)\rangle}$$
which simplifies to our desired formula.   
\end{proof}

\section{Computing the stringy $E$-function}\label{s:E-function}

Let $X = X(\DF)$ be a log-terminal horospherical variety of complexity one.
In Theorem~\ref{t:stringy-volume}, we compute the stringy motivic volume of $X$. As in the complexity zero case, see~\cite{Batyrev-Moreau}, Theorem~\ref{t:stringy-volume} requires $X$ to be $\QQ$-Gorenstein and log terminal; we introduce the related notions in~\ref{s:support}. In~\ref{s:desingularization} we construct a desingularization of $X$ in terms of its colored divisorial fan, and in~\ref{s:discrepancy}, we put the previous results to use by computing the discrepancy. This enables us to prove Theorem~\ref{t:stringy-volume} in~\ref{s:stringy}.
Throughout this section we follow the conventions of \ref{n:setting-arc}.

\subsection{Canonical class}\label{s:canonical}

The canonical class $K_X$ of $X$ can be expressed as a linear combination of $B$-invariant divisors, see \cite[Sections 2.3, 2.4]{LangloisTerpereau}. 
By the description in \cite[Section 16]{TimashevBook} there are two types of $G$-divisors called \emph{vertical} and \emph{horizontal}, depending on whether the $G$-action has complexity zero or one. The set of \emph{vertical} $G$-divisors on $X$ is parametrized by
\[
	\Vert (\DF) := \bigcup_{i \in J} \Vert(\D^i)
\]
where for any $i$ in $J$, $\Vert(\D^i)$ (equal to $\Vert(\D^i,\F^i)$) is the set of pairs 
$(y,p)$ such that $y$ is in $C^i$ and $p$ is a vertex of $\Delta_y^i$. 
The rest of the $G$-invariant divisors of $X$ are \emph{horizontal} and parametrized by
\[
	\Ray (\DF) := \bigcup_{i \in J} \Ray(\D^i,\F^i)
\]
where for any $i$ in $J$, we denote by $\Ray(\D^i,\F^i)$ the set of rays $\rho$ of $\sigma_i$ such that $\rho \cap \varrho(\F^i)$ is empty (namely the ray $\rho$ is uncolored) and 
$\rho \cap \deg (\D^i)$ is empty if $C^i$ is projective.
Here we used the same notation $\rho$ for a ray and for the corresponding primitive generator.

\begin{theorem}[{\cite[Theorem 2.11]{LangloisTerpereau}}]\label{t:G-div}
	Let $\Div(\DF)$ denote the set of $G$-divisors of $X$. There exists a bijection 
	\[
		\begin{array}{ccc}
			\Vert(\DF) \bigsqcup \Ray(\DF) & \to & \Div(\DF) \\
		\end{array}
	\]
	where we denote by $D_{(y,p)}$ -- respectively $D_\rho$ the $G$-divisor
	corresponding to the pair $(y,p)\in \Vert (\DF)$ -- respectively $\rho\in \Ray (\DF)$.	
\end{theorem} 

We may now describe the canonical class of $X$ as a Weil divisor.
\begin{theorem}[{\cite[Theorem 2.18]{LangloisTerpereau}}]\label{t-candiv}
	With the same notation as in Theorem~\ref{t:G-div}, the divisor
	\[
		K_X = \sum_{(y,p) \in \Vert(\DF)} (\kappa(p) b_y + \kappa(p)-1) \cdot D_{(y,p)} - \sum_{\rho \in \Ray(\DF)}  D_\rho - \sum_{\alpha \in \Phi \setminus I} a_\alpha \cdot D_\alpha
	\]
	is a canonical divisor of $X$. Here $K_C = \sum_{y \in C} b_y \cdot [y]$ is a canonical divisor on $C$, and $a_\alpha := \sum_{\beta \in R^+ \setminus R_I} \langle \beta,\alpha^\vee \rangle$, where $R^+$ is the set of positive roots of $G$ and $R_I$ is the set of roots of $P=N_G(H)$.
\end{theorem}

\subsection{Cartier divisors and support functions}\label{s:support}
We recall the combinatorial description of invariant Cartier divisors on horospherical varieties of complexity one from~\cite[Section 17]{TimashevBook} and \cite[Corollary 2.17]{LangloisTerpereau}, via functions on hypercones called \emph{support functions}. We refer to~\cite{PetersenSuss} for the setting of torus actions. Then, we give a criterion for these varieties to have $\QQ$-Gorenstein or log terminal singularities.

Denote by $\CC(X)^{(B)}$ the set of $B$-eigenfunctions in the function field $\CC(X)$ of $X$, which identifies with a subset of the tensor product $\CC(C) \otimes_\CC \CC[M]$. The principal divisor associated with a $B$-eigenfunction $f \otimes \chi^m$ in $\CC(C) \otimes_\CC \CC[M]$ is given by
\[
	\div(f \otimes \chi^m) = \sum_{(y,p) \in \Vert(\DF)} \kappa(p) \left( \langle m,p \rangle + \ord_y(f) \right) \cdot D_{(y,p)} + \sum_{\rho \in \Ray(\DF)} \langle m,\rho \rangle \cdot D_\rho + \sum_{D \in \F_{G/H}} \langle m,\varrho(D) \rangle \cdot D,
\]
where $\kappa(p) = \min \{ \lambda \in \ZZ_{>0} \mid \lambda \cdot p \in N \}$.
Cartier divisors are locally principal divisors. Invariant Cartier divisors on $X$ will be associated to some \emph{support functions} defined on the Cayley cones $C_y(\D^i)$.

\begin{definition}\label{d:support-fun} An \emph{integral linear function} on $\D^i$ is a map $\vartheta : C(\D^i) \to \QQ$ such that
\begin{enumerate}[(i)]
	\item for every $y \in C^i$ there exist $m_y \in M$ and $b_y \in \ZZ$ such that
	for any $(\nu,\ell) \in C_y(\D^i)$
		\[
			\vartheta(y,\nu,\ell) = \langle m_y,\nu \rangle +\ell c_y.
		\]
		
	\item if $C^i$ is projective, there exists $m \in M$ such that $m_y=m$ for any $y \in C$, and $f \in \CC(C)^*$ such that
		\[
			\div f = \sum_{y \in C} c_y \cdot [y].
		\]
\end{enumerate}

Denote by $\F_\DF$ the reunion of all the sets $\F^i$ for $i\in J$. A \emph{colored integral piecewise linear function} on $\DF$ is a pair $\theta=(\vartheta,(r_\alpha))$, where $\vartheta : |\DF| \to \QQ$ is a function such that the restriction $\vartheta\mid_{C(\D^i)}$ is integral linear for every $i$ in $J$, $\vartheta\mid_{C(\D^i)}$ and $\vartheta\mid_{C(\D^j)}$ coincide on $C(\D^i) \cap C(\D^j)$ for all $i$ and $j$ in $J$, and where $(r_\alpha)$ is a sequence of integers with $\alpha$ running over simple roots in $\Phi \setminus I$ such that $D_\alpha \not \in \F_\DF$.
More generally we say that $\theta=(\vartheta,(r_\alpha))$ is a \emph{colored piecewise linear function} on $\DF$ if there exists $k \in \ZZ_{>0}$ such that $k\theta$ is a colored integral piecewise linear function. We denote by $\PL(\DF)$ (resp. $\PL(\DF,\QQ)$) the set of colored integral piecewise linear functions (resp. colored piecewise linear functions) on $\DF$. 
\end{definition}

Now the Cartier divisor associated with a colored piecewise linear function $\theta$ in $\PL(\DF)$ is 
\[
	D_\theta = \sum_{(y,p) \in \Vert(\DF)} \vartheta(y,\kappa(p)p,\kappa(p)) D_{(y,p)} +\sum_{\rho \in \Ray(\DF)} \vartheta(C,\rho,0) D_\rho +\sum_{D \in \F_\DF} \vartheta(C,\varrho(D),0) D + \sum_{D_\alpha \not \in \F_\DF} r_\alpha D_\alpha.
\]
\begin{proposition}[{\cite[Corollary 2.19]{LangloisTerpereau}}]\label{p:theta}
	The variety $X(\DF)$ is $\QQ$-Gorenstein (i.e. $K_X$ is $\QQ$-Cartier) if and only if there exists $\theta=(\vartheta,(r_\alpha))$ in $\PL(\DF,\QQ)$ such that the following conditions are satisfied.
	\begin{enumerate}[(i)]
		\item There exists a canonical divisor $K_C = \sum_{y \in C} b_y \cdot [y]$ on $C$ where for every $(y,p)$ in $\Vert(\DF)$ we have			\[
				\vartheta(y,\kappa(p)p,\kappa(p)) = \kappa(p) b_y +\kappa(p) -1.
			\]
		\item For every $\rho$ in $\Ray(\DF)$ we have $\vartheta(C,\rho,0) = -1$.
		\item For every $D_\alpha$ in $\F_\DF$ we have $\vartheta(C,\varrho(D_\alpha),0) = -a_\alpha$.
	\end{enumerate}
\end{proposition}

Let $\theta_X$ be the colored piecewise linear function on $X$ satisfying the conditions of Proposition~\ref{p:theta} and such that $r_\alpha=-a_\alpha$ for any $\alpha$ with $D_\alpha \not \in \F_\DF$. The uniqueness follows from
\cite[Equations (17.1-2)]{TimashevBook}.

The support function $\varpi$ used in Theorem~\ref{t:stringy-volume} is constructed by gluing the following linear functions.
\begin{lemma}\label{l:gorenstein}
	Let $(\D,\F)$ be a colored $\sigma$-polyhedral divisor on $(C,G/H)$ such that $Y=X(\D,\F)$ is $\QQ$-Gorenstein. Let $z$ be a point in the locus of $\D$, and denote by $X(C_z(\D),\F)$ the horospherical $(G\times \CC^*)$-variety associated with the pair $(C_z(\D),\F)$ and the horospherical homogeneous space $G/H \times \CC^*$.  	
	Then there exists a linear function $\omega_{Y,z}$ on $C_z(\D)$ such that
	\begin{enumerate}[(i)]
		\item $\omega_{Y,z}(\rho,0)=\vartheta_Y(z,\rho,0)$ for any uncolored ray $\rho$ of $\sigma$,
		\item $\omega_{Y,z}(\tau)=-1$ for any ray $\tau$ of $C_z(\D)$ not contained in $N_\QQ$,
		\item $\omega_{Y,z}(\varrho(D_\alpha),0) = -a_\alpha$ for any color $D_\alpha \in \F$.
	\end{enumerate}
	In particular, if the locus of $\D$ is affine, we have
	\[
		X(\D,\F) \text{ $\QQ$-Gorenstein} \Rightarrow X(C_z(\D),\F) \text{ $\QQ$-Gorenstein.}
	\]
\end{lemma}

\begin{proof}
	As $Y$ is $\QQ$-Gorenstein, by Proposition~\ref{p:theta} there exists on $(\D,\F)$ a colored piecewise linear function $\theta_Y=(\vartheta_Y,(r_\alpha))$  satisfying the conditions (i)-(iii) of the proposition. 
	We construct the function $\omega_{Y,z}$ by modifying $\theta_Y$. To do so we construct a function $f$ in $\CC(C)^*$ such that
		$-\ord_z(f) = b_z +1$,
	for any $(z,p)$ in $\Vert(\D)$, where $K_C = \sum_{y \in C} b_y \cdot [y]$ is a canonical divisor of $C$. 
	Since $C$ is smooth, any divisor on $C$ is locally principal. Hence there exists an open neighborhood $C'$ of $z$ in $C$ and a function $f$ in $\CC(C')^*$ equal to $\CC(C)^*$ such that $(K_C+[z])_{\mid C'} = (\div f^{-1})_{\mid C'}$.
	The divisor $D_{\theta_Y} + \div(f)$ is a Cartier $\QQ$-divisor in $X$. The restriction $\omega_{Y,z}$ of its support function to $C_z(\D)$ is the required function. Indeed, as required, using Proposition~\ref{p:theta} we get that
	\[
		\omega_{Y,z}(\kappa(p)p,\kappa(p)) = \vartheta(z,\kappa(p)p,\kappa(p)) + \kappa(p) \ord_z(f)=-1.
	\]

	Let us now assume the locus of $\D$ is affine. The variety $X(C_z(\D),\F)$ being horospherical, it is $\QQ$-Gorenstein if and only if there exists a (restricted) linear function $\omega : C_z(\D) \to \QQ$ such that $\omega(\tau) = -1$ for any uncolored ray $\tau$ of $C_z(\D)$ (i.e. $\tau \cap \varrho(\F) = \emptyset$) and 
	$\omega(\varrho(D_\alpha),0) = -a_\alpha$ for any color $D_\alpha$ in $\F$
	(see~\cite[Proposition 4.1]{Brion93:SphericalMori}).
	These conditions are satisfied by the function $\omega_{Y,z}$. Indeed, since the locus of $\D$ is affine, for any uncolored ray $\rho$ of $\sigma$ we have
	$\omega_{Y,z}(\rho,0)=\vartheta_Y(z,\rho,0)=-1$.
\end{proof}

In Theorem~\ref{t:stringy-volume} we will need to assume $X$ is log terminal to ensure the convergence of the integral of the stringy motivic volume. We recall the following criterion.

\begin{theorem}[{\cite{LiendoSuss}, \cite[Theorem 2.22]{LangloisTerpereau}}]\label{t:canonical}
	Suppose $X=X(\DF)$ is $\QQ$-Gorenstein. Then $X$ has only log terminal singularities if and only if for any $(\D,\F) \in \DF$, one of the following assertions holds.
	\begin{enumerate}[(i)]
		\item The locus of $\D$ is affine.
		\item The locus of $\D$ is the projective line $\PP^1$ and $\sum_{y \in \PP^1} \left( 1 - \frac{1}{\kappa_y} \right) < 2$, where
		\[
			\kappa_y := \max \left( \{ \kappa(p) \mid p \text{ is a vertex of } \Delta_y \} \right).
		\]
		Here the polyhedron $\Delta_y$ is defined by $\D=\sum_{y \in C_0} \Delta_y \cdot [y]$.
	\end{enumerate}
\end{theorem}

\subsection{Desingularization}\label{s:desingularization}

Here we explain how to desingularize $X$ in terms of combinatorial data. The desingularization involves three proper birational morphisms:

\begin{equation}\label{e:desingularization}
	\begin{tikzpicture}[description/.style={fill=white,inner sep=2pt},baseline=(current  bounding  box.center)]
			\matrix (m) [matrix of math nodes, row sep=3em,column sep=2.5em, text height=1.5ex, text depth=0.25ex]
		{ 	X(\DF') & X(\tilde \DF) & X(\DFd) & X(\DF)=X. \\
		};
		\path[-]
			(m-1-1) edge[->] node[auto] {$q'$} (m-1-2)
			(m-1-2) edge[->] node[auto] {$q$} (m-1-3)
			(m-1-3) edge[->] node[auto] {$\pi_{dis}$} (m-1-4)
			(m-1-1) edge[->,bend right=10,looseness=1.4] node[below] {$\psi$} (m-1-4);
	\end{tikzpicture}
\end{equation}
The discoloration morphism $\pi_{dis}$ and the affinization morphism $q$ have already been defined in Equations~\eqref{e:discolor} and~\eqref{e:Etilde}. 
Let us now construct the morphism $q'$. Let $\Sp (\tilde \DF)$ be the set of special points of $\tilde \DF$ and denote its elements by $y_1,\dots,y_r$. Up to refining the affine coverings $(C_j^i)_{j \in J_i}$ in the construction outlined in Notation~\ref{n:affine-fan}, we may assume that each element of $\tilde \DF$ contains at most one special point. For $1 \leq i \leq r$, we let $\tilde \DF_i$ be the set of all polyhedral divisors $\D$ in $\DF$ which have $y_i$ as a special point and we define a fan $\Sigma_i$ on $N_\QQ \oplus \QQ$ as the fan generated by the Cayley cones $C_{y_i}(\D)$ for any $\D$ in $\tilde \DF$. Clearly the fans $\Sigma_i$ all contain the tail fan $\Sigma(\tilde \DF)$ as a subfan. Now following \cite[Section 11.1]{Cox:ToricVarieties} we will consider a star subdivision of $\Sigma(\tilde \DF)$ and compatible star subdivisions of the $\Sigma_{i}$.

\begin{definition}
	Consider a fan $\Sigma$ in $N_\QQ$ and an element $\nu$ of $|\Sigma| \cap N$ in the support of $\Sigma$. The \emph{star subdivision} of $\Sigma$ associated with $\nu$ is the fan consisting of the cones
	\begin{enumerate}[(i)]
		\item $\sigma \in \Sigma$ such that $\nu \not \in \sigma$,
		\item $\Cone(\tau,\nu)$ for $\tau \in \Sigma$ such that $\nu \not \in \tau$ and there exists $\sigma \in \Sigma$ with $\{\nu\} \cup \tau \subset \sigma$.
	\end{enumerate}
	A star subdivision is a subdivision of the original fan $\Sigma$ (see~\cite[Lemma 11.1.3]{Cox:ToricVarieties}). 
\end{definition}
 We say a fan $\Sigma$ is \emph{smooth} if any cone in $\Sigma$ is generated by a subset of a lattice basis of $N$. This is equivalent to the toric variety $X_\Sigma$ being smooth. By \cite[Theorem 11.1.9]{Cox:ToricVarieties} any fan $\Sigma$ has a smooth refinement $\Sigma'$ containing every smooth cone of $\Sigma$ and obtained from $\Sigma$ by a sequence of star subdivisions. Geometrically this implies that there is a projective desingularization of toric varieties $X_{\Sigma'} \to X_\Sigma$.\\

Let $\Sigma'(\tilde \DF)$ be a smooth refinement of the tail fan $\Sigma(\tilde \DF)$, corresponding to a sequence of star subdivisions associated with an element $(\nu_1,\dots,\nu_s) \in N^s$. For any $1 \leq i \leq r$ define $\hat \Sigma_i$ as the fan obtained from $\Sigma$ after applying star subdivisions with respect to $((\nu_1,0),\dots,(\nu_s,0))$ in $(N \oplus \ZZ)^s$. Clearly $\Sigma'(\tilde \DF)$ is a smooth subfan of $\hat \Sigma_i$. Finally, there exists a refinement $\Sigma_i'$ of $\hat \Sigma_i$ such that $\Sigma_i'$ contains every smooth cone of $\hat \Sigma_i$, including in particular every cone of $\Sigma'(\tilde \DF)$.

For any $\D$ in $\tilde \DF_i$ consider the set $\Sigma_{i,\D}$ of cones $\tau$ in $\Sigma_i'$ such that $\tau \subset C_{y_i}(\D)$. For $\tau\in \Sigma_{i,\D}$ we denote by $\sigma_{\tau,i}$ the cone $\tau \cap (N_\QQ \times \{ 0 \})$, and by $\Delta_{\tau,i}$ the set $\tau \cap (N_\QQ \times \{ 1 \})$ when $\tau \not \subset N_\QQ \times \{ 0 \}$. Since $\tau \subset C_{y_i}(\D)$ it follows that $\tau \cap (N_\QQ \times \{ 1 \})$ is a $\sigma_{\tau,i}$-polyhedron. We define a $\sigma_{\tau,i}$-polyhedral divisor 
\[
	\D_{\tau,i} := \sum_{y \in C_0} \Delta_y \cdot [y]
\]
 where $C_0 \subset C$ is the locus of $\D$ and $\Delta_y = \sigma_{\tau,i}$ for any $y\in C_0 \setminus \{ y_i \}$. In addition, 
$\Delta_{y_i} = \sigma_{\tau,i}$ if $\tau \subset N_\QQ \times \{0\}$ and 
 $\Delta_{y_i} = \Delta_{\tau,i}$ otherwise.
The set $\tilde \DF_\D := \{ \D_{\tau,i} \mid \tau \in \Sigma_{i,\D} \}$ is a divisorial fan. Moreover, the $\CC$-algebras embeddings $A(C_0,\D) \subset A(C_0,\D_{\tau,i})$ for $\D_{\tau,i} \in \tilde \DF_D$ induce a proper $T$-equivariant birational morphism $V(\tilde \DF_\D) \to V(\D)$, see \cite[Theorem 12.13]{TimashevBook}. Since the set $\DF' := \bigcup_{\D \in \tilde \DF} \tilde \DF_\D$ is a divisorial fan, the previous morphisms glue together into a proper birational morphism $q' : X(\DF') \to X(\tilde \DF)$. To prove that $X(\DF')$ is indeed smooth we use \cite[Chap. II]{KKMS73}, \cite[Proof of 2.6]{LiendoSuss}, and \cite[Theorem 2.5]{LangloisTerpereau}.

This concludes our construction of the desingularization of $X$.

\subsection{Discrepancy}\label{s:discrepancy}

The aim of this section is to compute the discrepancy of the $G$-equivariant morphism $\psi : X' \to X$ from Equation~\eqref{e:desingularization}, whose definition we recall below. 
\begin{center}
	\begin{tikzpicture}[description/.style={fill=white,inner sep=2pt},baseline=(current  bounding  box.center)]
			\matrix (m) [matrix of math nodes, row sep=3em,column sep=2.5em, text height=1.5ex, text depth=0.25ex]
		{ 	X':=X(\DF') & X(\tilde \DF) & X(\DFd) & X. \\
		};
		\path[-]
			(m-1-1) edge[->] node[auto] {$q'$} (m-1-2)
			(m-1-2) edge[->] node[auto] {$q$} (m-1-3)
			(m-1-3) edge[->] node[auto] {$\pi_{dis}$} (m-1-4)
			(m-1-1) edge[->,bend right=10,looseness=1.4] node[below] {$\psi$} (m-1-4);
	\end{tikzpicture}
\end{center}
The exceptional divisors of $\psi$ are particular $G$-divisors of $X'$, that is, they correspond to some elements of $\Vert(\DF') \sqcup \Ray(\DF')$. 
\\
\paragraph{{\bf Exceptional divisors of $\pi_{dis}$.}} 
As $\Vert(\DFd)$ is equal to $\Vert(\DF)$ the exceptional divisors of $\pi_{dis}$ are in one-to-one correspondence with the elements of $\Ray(\DFd)$ which do not come from $\Ray(\DF)$, i.e., the elements of 
\begin{equation} \label{ray}
	\bigcup _{i \in J} \Ray(\D^i,\emptyset) \setminus \Ray(\D^i,\F^i).
\end{equation}
By the lemma, the associated divisors $D_{\rho_1},\dots,D_{\rho_t} \subset X(\DFd)$ in (\ref{ray}) are the exceptional divisors of $\pi_{dis}$.
\\
\paragraph{{\bf Exceptional divisors of $q$.}}
The morphism $q : X(\tilde \DF) \to X(\DFd)$ is obtained via parabolic induction on the $T$-equivariant morphism $V(\tilde \DF) \to V(\DF)$ from Equation~\eqref{e:Etilde}, so the exceptional divisors of $q$ are the parabolic inductions of the exceptional divisors of the latter map. By \cite[Theorem 10.1]{AltmannHausen}, the set $\Ray(\tilde \DF) \setminus \Ray(\DFd)$ is equal to the reunion over all colored polyhedral divisors $(\D,\emptyset)$ in $\tilde \DF$ with projective locus, of the rays $\rho$ of the tails of polyhedral divisors $\D$ with the property that $\deg \D \cap \rho \neq \emptyset$. We denote the associated exceptional divisors of $q$ by $D_{\rho_{t+1}},\dots,D_{\rho_s} \subset X(\tilde \DF)$.
\\
\paragraph{{\bf Exceptional divisors of $q'$.}}
As before the map $q' : X'=X(\DF') \to X(\tilde \DF)$ is obtained from the map $q'' : V(\DF') \to V(\tilde \DF)$. Denote by $\{ y_1,\dots,y_u \}$ the set of special points of $\tilde \DF$. For $1 \leq i \leq u$ consider the fans $\Sigma_i$ and $\Sigma_i'$ constructed in Section~\ref{s:desingularization}. The exceptional divisors of $q''$ (or equivalently of $q'$) are in bijection with the triples $[y_i,\nu,\ell]$, where $(\nu,\ell)$ is a ray of $\Sigma_i' \setminus \Sigma_i$. When $\ell = 0$ the triple $(y_i,\nu,0)$ corresponds to an element of $\Ray(\DF')$, otherwise it corresponds to an element of $\Vert(\DF')$. We denote the associated exceptional divisors by $D_{\rho_{s+1}}',\dots,D_{\rho_e}'$ and $D_{(y_1,p_1)}',\dots,D_{(y_u,p_u)}'$, respectively.
\\
\paragraph{{\bf Relative canonical class of $\psi$.}}
Since $\Sigma(\tilde \DF) = \Sigma (\DFd) = \Sigma (\DF)$ and $\Sigma(\DF')$ is a fan subdivision of $\Sigma(\tilde \DF)$, we may identify the pullbacks of $D_{\rho_1},\dots,D_{\rho_s}$ with the corresponding divisors $D_{\rho_1}',\dots,D_{\rho_s}'$ in $X'$. 
By the previous paragraphs, the exceptional divisors of $\psi$ are the $(D_{\rho_i}')_{1 \leq i \leq e}$ and the $(D_{(y_j,p_j)}')_{1 \leq j \leq u}$.

\begin{proposition}\label{p:relative}
	If $X(\DF)$ is $\QQ$-Gorenstein, then the relative canonical class of $\psi$
	is represented by
        \begin{align*}
		K_{X' / X} &= \sum_{j=1}^u \left( \kappa(p_j)b_{y_j}+\kappa(p_j)-1-\vartheta_X (y,\kappa(p_j)p_j,\kappa(p_j)) \right)  D_{(y_j,p_j)}' 
		&+ \sum_{i=1}^e \left(-1-\vartheta_X (C,\rho_i,0) \right)  D_{\rho_i}' ,
	\end{align*}
	where $\sum_{y \in C} b_y \cdot [y]$ is a canonical divisor of $C$, $\theta_X = (\vartheta_X,(r_\alpha))$ is the support function from Section~\ref{s:support}, and $\kappa : N_\QQ \to \ZZ$ is defined by 	
		$\kappa(p) = \min \{ \lambda \in \ZZ_{>0} \mid \lambda \cdot p \in N \}$.
\end{proposition}

\begin{proof}
	A canonical divisor of $X'$ is given by 
	\[
		K_{X'} = \sum_{(y,p) \in \Vert(\DF')} \vartheta_{X'}(y,\kappa(p)p,\kappa(p))  D_{(y,p)}' + \sum_{\rho \in \Ray(\DF')}  \vartheta_{X'}(C,\rho,0) D_\rho' - \sum_{\alpha \in \Phi \setminus I} a_\alpha  D_\alpha',
	\]
	while the pullback by $\psi$ of a canonical divisor of $X$ is
	\[
		\psi^* K_{X} = \sum_{(y,p) \in \Vert(\DF')} \vartheta_X(y,\kappa(p)p,\kappa(p))  D_{(y,p)}' + \sum_{\rho \in \Ray(\DF')}  \vartheta_X(C,\rho,0)  D_\rho' - \sum_{\alpha \in \Phi \setminus I} a_\alpha  D_\alpha',
	\]
	where we choose $K_X$ as in Theorem~\ref{t:canonical}. Hence
	\[
		K_{X' / X} = \sum_{j=1}^u (\vartheta_{X'}-\vartheta_X) (y,\kappa(p_j)p_j,\kappa(p_j))  D_{(y_j,p_j)}' + \sum_{i=1}^e (\vartheta_{X'}-\vartheta_X) (C,\rho_i,0)  D_{\rho_i}'.
	\]
	We conclude by computing the values of $\vartheta_{X'}$ using Proposition~\ref{p:theta}.
\end{proof} 

\subsection{Stringy invariants}\label{s:stringy}
This section is devoted to determine the stringy motivic volume of a log terminal
horospherical variety of complexity one by using the discrepancy calculation from Section~\ref{s:discrepancy}. We use the notations of the previous section, and  denote by $\Gamma \subset C$ an open dense affine subset which does not contain any special point. Let $d$ be the dimension of $X$.
\\
\paragraph{{\bf Decomposition of the motivic integral.}}
We decompose the motivic volume $\E(X)$ along the fibers of the map $\val : X_\Gamma' \to |\DF'|_\Gamma \cap \N$ from Equation~\eqref{e:val}.
The next statement immediately follows from the measurability results of 
 Lemmas~\ref{l:vol-hor} and~\ref{l:vol-ver}.

\begin{lemma}\label{l:E-decomposition}
	With above notations, the stringy motivic volume can be decomposed as 
	\[
		\E(X) = \sum_{(y,\nu,\ell) \in |\DF'|_\Gamma \cap \N } \int_{\val^{-1}(y,\nu,\ell)} \LL^{-\ord_{K_{X'/X}}} d\mm_{X'}.
	\]
\end{lemma}
\paragraph{{\bf The stringy support function.}}
Now we want to calculate the right-hand side of the identity in Lemma~\ref{l:E-decomposition}, depending on whether $\ell=0$ or $\ell \geq 1$. As we did in Lemma~\ref{l:gorenstein} for a single hypercone, we need to introduce a new support function $\omega_X$ inspired from the definition of the function $\theta_X$.

\begin{propdef}\label{pd:support}
	Denote by $(\D^i,\F^i)$ the elements of $\DF$ for $i \in J$. There exists a pair $\varpi_X=(\omega_X,(r_{\alpha}))$, where $\omega_X : |\DF| \to \QQ$ is a function and $r_\alpha \in \ZZ$ for any $\alpha \in \Phi \setminus I$, such that
	\begin{itemize}
		\item for any $i \in J$ and any $y$ in the locus of $\D^i$, we have
			\begin{enumerate}[(i)]
				\item $\omega_X(y,\rho,0)=\vartheta_X(y,\rho,0)$ for any uncolored ray $\rho$ of the tail of $\D^i$,
				\item $\omega_X(y,\tau)=-1$ for any ray $\tau$ of $C_y(\D^i)$ not contained in $N_\QQ$,
				\item $\omega_X(y,\varrho(D_\alpha),0)=-a_\alpha$ for any color $D_\alpha \in \F^i$,
				\item $r_\alpha=-a_\alpha$ for any $\alpha \in \Phi \setminus I$ such that $D_\alpha \not \in \F_\DF$, 
				\item $\omega_X(y,\nu,\ell) = \langle m_y^i,\nu \rangle + \ell c_y^i$ for some $m_y^i \in M$ and $c_y^i \in \QQ$.
			\end{enumerate}
		\item If $(y,\nu,\ell) \in C(\D^i) \cap C(\D^j)$, then 
			$\langle m_y^i,\nu \rangle + \ell c_y^i = \langle m_y^j,\nu \rangle + \ell c_y^j.$
	\end{itemize}
	We call the pair $\varpi_X=(\omega_X,(r_\alpha))$ the \emph{stringy support function}.
\end{propdef}

\begin{proof}
	Since $X$ is $\QQ$-Gorenstein, by Lemma~\ref{l:gorenstein}, for any $i \in J$ and $y$ in the locus of $\D^i$, there exists a linear function $\omega_{i,y}$ on $C_y(\D^i)$ satisfying Conditions (i)-(iii) above. It is enough to check that $\omega_{i,y}$ and $\omega_{j,y}$ coincide on $C_y(\D^i) \cap C_y(\D^j)$. 
	Indeed $C_y(\D^i) \cap C_y(\D^j)$ is a common face to both $C_y(\D^i)$ and $C_y(\D^j)$. It generated by colored or uncolored rays. Moreover both $\omega_{i,y}$ and $\omega_{j,y}$ are linear on the common face and coincide on the rays, which concludes the proof.
\end{proof}

\paragraph{{\bf Motivic volume for horizontal arcs.}} 
We compute the motivic integral over subsets of the form $\val^{-1}(C,\nu,0)$.

\begin{lemma}\label{l:integral-hor}
	Let $(y,\nu,0)$ be an element of $|\DF'|_\Gamma=|\DF|_\Gamma$. We have
	\[
		\int_{\val^{-1}(C,\nu,0)} \LL^{-\ord_{K_{X'/X}}} d\mm_{X'} = [G/H] [\Gamma] \LL^{\vartheta_X(C,\nu,0)}.
	\]
\end{lemma}

\begin{proof}
	We only prove the result for a $T$-variety. The general case is obtained by parabolic induction. Let $\sigma \in \Sigma(\DF')$ be a cone containing $\nu$ and $\D$ be a $\sigma$-polyhedral divisor of $\DF'$ with (affine) locus $C_0$. The rays of $\sigma$ either correspond to exceptional divisors of $\psi : X' \to X$, or are elements of $\Ray(\DF)$. For the exceptional divisors we use the notation of Section~\ref{s:discrepancy}. Up to renumbering we may assume that the exceptional rays of $\sigma$ are $\rho_1,\dots,\rho_a$, and we denote the remaining rays by $\tau_{a+1},\dots,\tau_r$. Write $V_\QQ := \sigma^\vee \cap (-\sigma^\vee)$ and choose a basis $(v_1,\dots,v_{d-1-r})$ of $V_\QQ \cap M$. Let $u_1,\dots,u_a,u_{a+1},\dots,u_r$ be the duals of $\rho_1,\dots,\rho_a,\tau_{a+1},\dots,\tau_r$. 
	
	From \cite[Remark 3.16 (2)]{PetersenSuss} we know that the ideal of a $T$-stable divisor $D_{\rho_j}'$ is given by
	\begin{equation}\label{e:ideal-ray}
		I(D_{\rho_j}') = \bigoplus_{m \in (\sigma^\vee \setminus \rho_j^\perp) \cap M} H^0(C_0,\OO_{C_0}(\lfloor \D(m) \rfloor)) \otimes \chi^m.
	\end{equation}
	If $\alpha \in \val^{-1}(C,\nu,0)$ and $f \otimes \chi^m$ is a homogeneous element of degree $m$ in $A(C_0,\D)$, we have
	\[
		\alpha^*(f\otimes \chi^m) = t^{\nu(m)} \omega(m) \qfun(f),
	\]
	where we use the notation of Remark~\ref{r:data-arc}. Hence 
	$
		\sup \left\{ k \in \NN \mid  \alpha^*(f\otimes \chi^m) \in (t^k) \right\} = \langle m,\nu \rangle.
	$
	
	Write 
        $
		\nu = \nu_1 \rho_1 + \dots + \nu_a \rho_a + \nu_{a+1} \tau_{a+1} + \dots +\nu_r \tau_r.
	$
	We have $\nu_i \in \NN$ for $1 \leq i \leq r$. An element $m$ of $(\sigma^\vee \setminus \rho_j^\perp) \cap M$ can be decomposed as 
	$
		m=m_1 u_1 + \dots + m_r u_r + m_1' v_1 + \dots + m_{d-1-r}' v_{d-1-r}
	$
	with $m_i \in \NN$ for $1 \leq i \leq r$, $m_j \neq 0$, and $m_1', \dots, m_{d-1-r}'$ in $\ZZ$. Hence $\langle m,\nu \rangle = \sum_{i=1}^r m_i \langle u_i,\nu \rangle$. Finally we have
	\[
		\min \{ \langle m,\nu \rangle \mid m \in (\sigma^\vee \setminus \rho_j^\perp) \cap M \} = \langle u_j,\nu \rangle.
	\]
	We have $\ord_{D_{\rho_j}'} \alpha = \langle u_j,\nu \rangle$ and $\ord_{D_{(y_i,p_i)}'} \alpha =0$ since $\alpha \in \val^{-1}(C,\nu,0)$.  We obtain
	\begin{equation}\label{e:vanishing-order}
		\ord_{K_{X'/X}}(\alpha) = \sum_{i=1}^a (-1-\vartheta_X(C,\rho_i,0)) \langle u_i,\nu \rangle.
	\end{equation}
	Now by Equation~\eqref{e:vanishing-order} and Theorem~\ref{t:volume} we get 
	\begin{align*}
		\int_{\val^{-1}(C,\nu,0)} \LL^{-\ord_{K_{X'/X}}} d\mm_{X'} &=\mu_{X'}(\val^{-1}(C,\nu,0)) \LL^{\sum_{i=1}^a (1+\vartheta_X([C,\rho_i,0])) \langle u_i,\nu \rangle} \\
		&= [\Gamma] [G/H] \LL^{-\sum_{j=1}^r \langle u_j,\nu \rangle + \sum_{i=1}^a (1+\vartheta_X(C,\rho_i,0)) \langle u_i,\nu \rangle}.
	\end{align*}
	
	By linearity of $\vartheta_X$ on each Cayley cone and the decomposition of $\nu$ on the integral basis, we obtain
	\[
		\vartheta_X(C,\nu,0) = \sum_{i=1}^a \langle u_i,\nu \rangle \vartheta_X(C,\rho_i,0) + \sum_{j=a+1}^r \langle u_j,\nu \rangle \vartheta_X(C,\tau_j,0).
	\]
	Finally, by definition of $\vartheta_X$, we have $\vartheta_X(C,\tau_j,0)=-1$ for $a+1 \leq j \leq r$, which concludes the proof.
\end{proof}

\begin{corollary}\label{c:integral-hor}
	With the notations of Lemma~\ref{l:integral-hor}, we have
	\[
		\int_{\val^{-1}(C,\nu,0)} \LL^{-\ord_{K_{X'/X}}} d\mm_{X'} = [G/H] [\Gamma] \LL^{\omega_X(C,\nu,0)}.
	\]
\end{corollary}

\paragraph{{\bf Motivic volume for vertical arcs.}} 
We now consider motivic integrals over vertical components of the arc space $\L_\Gamma(X')$. 
Let $\xi = (y,\nu,\ell)$ be an element of $|\DF'|_\Gamma \cap \N$ such that $\ell \geq 1$.
As in the previous lemma we assume $G=T$ and $H=\{e\}$. Let $\D$ be a polyhedral divisor of $\DF'$ with (affine) locus $C_0$ such that $(\nu,\ell) \in C_y(\D)$, and denote by $\sigma$ the tail of $\D$. Let $r$ be the dimension of $\sigma$ and $s$ be the dimension of $C_y(\D)$.
	Up to renumbering the exceptional divisors of Section~\ref{s:discrepancy}, we may assume that the rays of $\sigma$ either correspond to exceptional divisors $D_{\rho_1}',\dots,D_{\rho_a}'$ of $\psi : X' \to X$, or are elements $\tau_{a+1},\dots,\tau_r$ of $\Ray(\DF)$. Similarly the other rays 
	\[
		\lambda_1 := (\kappa(p_1)p_1,\kappa(p_1)),\dots,\lambda_{s-r} := (\kappa(p_{s-r})p_{s-r},\kappa(p_{s-r}))
	\]
	of $C_y(\D)$ include those associated with the exceptional divisors $D_{(y,p_1)}',\dots,D_{(y,p_{u})}'$.
	
	Let $\Lambda = M \oplus \ZZ$. Write $V_\QQ := C_y(\D)^\vee \cap (-C_y(\D)^\vee)$ and choose a basis $(v_1,\dots,v_{d-s})$ of $V_\QQ \cap \Lambda$. Let $u_1,\dots,u_a,u_{a+1},\dots,u_r$ be the duals of $\rho_1,\dots,\rho_a,\tau_{a+1},\dots,\tau_r$ in $N$ and $w_1,\dots,w_{s-r}$ be the duals of $\lambda_1,\dots,\lambda_{s-r}$. A set of generators of the semigroup $C_y(\D)^{\vee} \cap \Lambda$ is given by
	\[
		(u_1,0),\dots,(u_r,0) ; \pm v_1,\dots,\pm v_{d-s} ; w_1,\dots, w_{s-r}.
	\]

\begin{lemma}\label{l:ideal-vert}
	 We have
	\[
		\int_{\val^{-1}(\xi)} \LL^{-\ord_{K_{X'/X}}} d\mm_{X'} = [G/H] (\LL-1) \LL^{\vartheta_X(\xi)-\sum_{i=1}^{s-r} \left( \kappa(p_i) b_y + \kappa(p_i) \right) \langle w_i , (\nu,\ell) \rangle}.
	\]
\end{lemma}

\begin{proof}
	Let $\pi$ in $\CC(C)^*$ be a uniformizer of $y$. Recalling the expression of the ideal $I(D_{\rho_j}')$ from Equation~\eqref{e:ideal-ray} we can compute $\ord_{D_{\rho_j}'}(\alpha)$ by studying $\alpha^*(\pi^k \otimes \chi^m)$ for all $(m,k)$ in $(C_y(\D)^\vee \setminus (\rho_j,0)^\perp) \cap \Lambda.$
	We may then use the same argument as in Lemma~\ref{l:integral-hor}. We obtain that
		$\ord_{D_{\rho_j}'} (\alpha) = \langle u_j,\nu \rangle$
	for any $\alpha$ in $\val^{-1}(\xi)$. Let us now recall the description of $I(D_{(y,p_i)}')$ from \cite[Remark 3.16 (i)]{PetersenSuss}:
	\[
		I(D_{(y,p_i)}') = \bigoplus_{m \in \sigma^\vee \cap M} H^0(C_0,\OO_{C_0}(\lfloor \D(m) \rfloor)) \cap \{ f \in \CC(C_0) \mid \ord_y f + \langle m,p_i \rangle > 0 \} \otimes \chi^m.
	\]
	As before we can compute $\ord_{D_{(y,p_i)}'}(\alpha)$ from the computation of $\alpha^*(\pi^k \otimes \chi^m)$ for any pair $(m,k)$ in $(C_y(\D)^\vee \setminus (p_i,1)^\perp) \cap \Lambda$. We find $\ord_{D_{(y,p_i)}'}(\alpha) = \langle w_i,(\nu,\ell) \rangle.$
	Then it follows from Proposition~\ref{p:relative} that
	\begin{equation}\label{e:vanishing-order-ver}
		\ord_{K_{X'/X}} (\alpha) = \sum_{j=1}^a c_j \langle u_j,\nu \rangle + \sum_{i=1}^{s-r} d_i \langle w_i,(\nu,\ell) \rangle
	\end{equation}
	where $c_j = -1-\vartheta_X(C,\rho_j,0)$ and $d_i = \kappa(p_i)b_y + \kappa(p_i) - 1 - \vartheta_X(y,\kappa(p_i)p_i,\kappa(p_i))$.
	Note that $d_{i}= 0$ for $u<i\leq s-r$. 
	Now by Equation~\eqref{e:vanishing-order-ver} and Theorem~\ref{t:volume},
	\begin{align*}
		\int_{\val^{-1}(\xi)} \LL^{-\ord_{K_{X'/X}}} d\mm_{X'} 
		&=\mu_{X'}(\val^{-1}(\xi)) \LL^{-\sum_{j=1}^a c_j \langle u_j,\nu \rangle 
		- \sum_{i=1}^{s-r} d_i \langle w_i,(\nu,\ell) \rangle} 	\\
		&= [G/H] (\LL-1) \LL^{-\sum_{j=1}^r \langle u_j,\nu \rangle -\sum_{i=1}^{s-r} \langle w_i,(\nu,\ell)\rangle-\sum_{j=1}^a c_j \langle u_j,\nu \rangle 
		- \sum_{i=1}^{s-r} d_i \langle w_i,(\nu,\ell) \rangle}.
	\end{align*}
	Since $\vartheta_X$ is linear on Cayley cones, decomposing $\nu$ on the integral basis, we conclude using the identity
	\begin{align*}
		\vartheta_X(\xi) -\sum_{i=1}^{s-r} \left( \kappa(p_i) b_y + \kappa(p_i) \right) \langle w_i , (\nu,\ell) \rangle &
		= \sum_{j=1}^a (-1-c_j) \langle u_j,\nu \rangle - \sum_{j=a+1}^r \langle u_j,\nu \rangle - \sum_{i=1}^{s-r} (1+d_i) \langle w_i , (\nu,\ell) \rangle.
	\end{align*}
\end{proof}
To rephrase the formula of Lemma~\ref{l:ideal-vert} in terms of the function $\omega_X$ we need  
\begin{proposition}\label{p:identity-vert}
	Let $\DF$ be a colored divisorial fan on $(C,G/H)$. Assume $X:=X(\DF)$ is $\QQ$-Gorenstein and let $X'=X(\DF')$ be the desingularization of $X$ as defined in Equation~\eqref{e:desingularization}. Then for any $(y,p)$ in $\Vert(\DF') \setminus \Vert(\DF)$ we have
	$
		(\vartheta_{X'}-\vartheta_X)(y,\kappa(p)p,\kappa(p)) = -1 - \omega_X(y,\kappa(p)p,\kappa(p)).
	$
\end{proposition}

\begin{proof}
	Since the problem is local we may assume that $\DF=\{ (\D,\F) \}$.
We first show the statement when $\DF = \tilde \DF$, that is, when the coloration $\F$ is trivial and the locus $C_{0}$ of $\D$ is affine. By parabolic induction we may also assume that $G=T$ and $H=\{e\}$. Write
	\[
		\D = \sum_{z \in C_0} \Delta_z \cdot [z]
	\]
	and let $y\in C_0$. Denote $\DF' = \{\D^i\}_{i \in J}$.
	The cones $C_y(\D^i)$ generate a fan $\Sigma'$ in $N_\QQ \oplus \QQ$. Denote by $X_{tor}$ (resp. by $X_{tor}'$) the toric $(T\times \CC^*)$-variety associated with the cone $C_y(\D)\subset N_\QQ \oplus \QQ$ (resp. with the fan $\Sigma'$). By Lemma~\ref{l:gorenstein} we know that $X_{tor}$ is $\QQ$-Gorenstein. Moreover $X_{tor}'$ is smooth, and the proper birational map $X_{tor}' \to X_{tor}$ is a $(T \times \CC^*)$-equivariant resolution of singularities.
	
	Define a $\sigma$-polyhedral divisor $\hat \D$ with locus $\AA^1$ on $(\PP^1,T)$ by setting $\hat \D := \sum_{z \in \AA^1} \hat \Delta_z \cdot [z]$, where
	\begin{equation*}
		\hat \Delta_z = \begin{cases}
							\Delta_y & \text{if $z=0$,} \\
							\sigma & \text{otherwise.}
						\end{cases}
	\end{equation*}
	Similarly, for any $i \in J$, we define a polyhedral divisor $\hat \D^i$ with locus $\AA^1$ associated with $\D^i$, and we let $\hat \DF'$ be the corresponding divisorial fan.
	We have a commutative diagram
	\begin{center}
		\begin{tikzpicture}[description/.style={fill=white,inner sep=2pt}]
			\matrix (m) [matrix of math nodes, row sep=3em,column sep=2.5em, text height=1.5ex, text depth=0.25ex]
		{ 	X_{tor}' & X_{tor} \\
			X(\hat \DF') & X(\hat \D). \\
		};
		\path[-]
			(m-1-1) edge[->] node[auto] {$q'$} (m-1-2)
			(m-2-1) edge[->] node[auto] {$q'$} (m-2-2)
			(m-1-1) edge[double,double distance=3pt] (m-2-1);
		\path[-]
			(m-1-2) edge[double,double distance=3pt] (m-2-2);
		\end{tikzpicture}
	\end{center}
	Indeed, the equality on the right comes from the description of the $\CC$-algebra $A(\AA^1,\hat \D)$ viewed as the semigroup algebra $\CC[C_0(\hat \D)^\vee \cap (M \oplus \ZZ)]$. The other equality is similar.

	Consider the support functions $\theta_X=(\vartheta_X,(-a_\alpha))$ in $\PL(\D)$ and $\theta_{X'}=(\vartheta_{X'},(-a_\alpha))$ in $\PL(\DF')$ satisfying the hypotheses of Proposition~\ref{p:theta}, which induce functions respectively on $C_y(\D)$  and $\Sigma'$. To prove the proposition we compute the discrepancy of $q'$ in two different ways.
	
	Let us first compute the discrepancy of $q' : X(\hat \DF') \to X(\hat \D)$ using the functions $\vartheta_{X'}$ and $\vartheta_X$. By definition the Weil $\QQ$-divisor
	\begin{align*}
		\hat K &:= \sum_{(0,p) \in \Vert(\hat \DF') \setminus \Vert(\hat \D)} (\vartheta_{X(\hat \DF')} -\vartheta_{X(\hat \D)}) (0,\kappa(p)p,\kappa(p)) D_{(0,p)}' + \sum_{\rho \in \Ray(\hat \DF') \setminus \Ray(\hat \D)} (\vartheta_{X(\hat \DF')} -\vartheta_{X(\hat \D)})(\PP^{1},\rho,0) D_\rho' \\
		&= \sum_{(y,p) \in \Vert(\DF') \setminus \Vert(\D)} (\vartheta_{X'} -\vartheta_{X}) (y,\kappa(p)p,\kappa(p)) D_{(0,p)}' + \sum_{\rho \in \Ray(\DF') \setminus \Ray(\D)} (\vartheta_{X'} -\vartheta_{X})(C,\rho,0) D_\rho'
	\end{align*}
	is a relative canonical divisor.
	Next we compute the discrepancy of $q' : X_{tor}' \to X_{tor}$ using the function $\omega_{X}$. By \cite[Proposition 4.2]{Batyrev-Moreau} the Weil $\QQ$-divisor
	\[
		K := \sum_{(y,p) \in \Vert(\DF') \setminus \Vert(\D)} (-1-\omega_X(y,\kappa(p)p,\kappa(p))) D_{(0,p)}'  + \sum_{\rho \in \Ray(\DF') \setminus \Ray(\D)} (-1 -\omega_{X}(y,\rho,0)) D_\rho'
	\]
	is a relative canonical divisor. The divisors $K$ and $\hat K$ are linearly equivalent. By \cite[Lemma 3.39]{KollarMori:IntroMMP} we know that they are in fact equal. Thus
	\begin{equation}\label{e:vartheta-omega}
		-1-\omega_X(y,\kappa(p)p,\kappa(p)) = (\vartheta_{X(\hat \DF')} -\vartheta_{X(\hat \D)}) (0,\kappa(p)p,\kappa(p))
\end{equation}
	for any $(y,p) \in \Vert(\DF') \setminus \Vert(\D)$, which concludes the case where $\DF = \tilde \DF$.
	
	For the case where $\DF$ is general, we remark that the functions $\omega_{X(\tilde \DF)} + \vartheta_{X} - \vartheta_{X(\tilde \DF)}$ and $\omega_{X}$ coincide. Indeed, they are linear on each Cayley cone of $\DF$ and they have the same values on the rays. Hence taking this into account we obtain Equation \eqref{e:vartheta-omega} in the general case, which concludes the proof. 	
\end{proof}

\begin{corollary}\label{c:integral-ver}
	With the notations of Lemma~\ref{l:ideal-vert}, when $\ell \geq 1$, we have
	\[
		\int_{\val^{-1}(y,\nu,\ell)} \LL^{-\ord_{K_{X'/X}}} d\mm_{X'} = [G/H](\LL-1) \LL^{\omega_{X}(y,\nu,\ell)}.
	\]
\end{corollary}

\begin{proof}
	We proved in Lemma~\ref{l:ideal-vert} that
	\begin{equation}\label{e:integral}
		\int_{\val^{-1}(y,\nu,\ell)} \LL^{-\ord_{K_{X'/X}}} d\mm_{X'} = [G/H] (\LL-1) \LL^{\vartheta_X(y,\nu,\ell)-\sum_{i=1}^{s-r} \left( \kappa(p_i) b_y + \kappa(p_i) \right) \langle w_i , (\nu,\ell) \rangle}.
	\end{equation}
	By definition of $\vartheta_{X'}$ and $\vartheta_X$, we have
	\[
		\vartheta_X(y,\nu,\ell)-\sum_{i=1}^{s-r} \left( \kappa(p_i) b_y + \kappa(p_i) \right) \langle w_i , (\nu,\ell) \rangle = (\vartheta_X-\vartheta_{X'})(y,\nu,\ell)-\sum_{i=1}^r \langle \nu,u_i \rangle -\sum_{j=1}^r \langle (\nu,\ell),w_j \rangle.
	\]
	Moreover by linearity
	\begin{equation}\label{e:identity-theta1}
		(\vartheta_{X}-\vartheta_{X'})(y,\nu,\ell) = \sum_{i=1}^r \langle \nu,u_i \rangle (\vartheta_X-\vartheta_{X'})(y,\rho_i,0) + \sum_{j=1}^{s-r} \langle (\nu,\ell),w_j \rangle (\vartheta_X-\vartheta_{X'})(y,\kappa(p_j)p_j,\kappa(p_j)).
	\end{equation}
	By the proof of Corollary~\ref{c:integral-hor} and Proposition~\ref{p:identity-vert} we have
	$$(\vartheta_X-\vartheta_{X'})(y,\rho_i,0) = \omega_X(y,\rho_i,0)-1 \:\:\text{and}\:\:
		(\vartheta_X-\vartheta_{X'})(y,\kappa(p_j)p_j,\kappa(p_j)) = 1 + \omega_X(y,\kappa(p_j)p_j,\kappa(p_j)).$$
	We obtain
	\begin{equation}\label{e:identity-theta2}
		(\vartheta_{X}-\vartheta_{X'})(y,\nu,\ell) - \sum_{i=1}^r \langle \nu,u_i \rangle - \sum_{j=1}^{s-r} \langle (\nu,\ell),w_j \rangle = \omega_X(y,\nu,\ell).
	\end{equation}
	Replacing the exponent of $\LL$ in the right-hand side of Equation~\eqref{e:integral} by $\omega_X(y,\nu,\ell)$ using Equations~\eqref{e:identity-theta1} and~\eqref{e:identity-theta2}, we obtain the stated result.
\end{proof}

\paragraph{{\bf The stringy motivic volume.}} 

\begin{theorem}\label{t:stringy-volume}
	Let $\DF$ be a colored divisorial fan on $(C,G/H)$ such that $X=X(\DF)$ is $\QQ$-Gorenstein with log terminal singularities. Then for any open dense subset $\Gamma$ in $C \setminus \Sp(\DF)$ we have
	\[
		\E(X) = [G/H] \sum_{(y,\nu,\ell) \in |\DF|_\Gamma \cap \N} [X_\ell]\, \LL^{\omega_X(y,\nu,\ell)},
	\]
	where $X_0=\Gamma$ and $X_\ell=\AA^1\setminus \{0\}$ if $\ell \geq 1$. 
        The stringy $E$-function of $X$ is computed as follows
	\[
		\EE(X;u,v) = E(G/H;u,v) \sum_{(y,\nu,\ell) \in |\DF|_\Gamma \cap \N} E(X_\ell;u,v)\, (uv)^{\omega_X(y,\nu,\ell)},
	\]
	where $E(X_\ell;u,v) = E(\Gamma;u,v)$ if $\ell=0$, else $E(X_\ell;u,v)=uv-1$.
\end{theorem}

\begin{proof}
This follows from Lemma~\ref{l:E-decomposition} and Corollaries~\ref{c:integral-hor} and~\ref{c:integral-ver}.
\end{proof}
\paragraph{{\bf Rational form and candidate poles.}}
We end this section by giving a rational expression of the stringy motivic volume $\E(X)$ in terms of the combinatorial data $\DF$.
For a special point $y$ in $\Sp(\DF)$, we denote by $\DF_{y}$ the fan generated by the Cayley cones of the form $C_{y}(\D^{i})$, where $i\in J$.
Let $\tau$ belong to $\DF_{y}$. Fix a fan $\Sigma_{\tau}$ with support $\tau$ such that
every cone of $\Sigma_{\tau}$ is simplicial and the cones of dimension one in $\Sigma_{\tau}$ are exactly the faces of dimension one of $\tau$. Such a fan always exists (see \cite[V.4]{Ewa96}). 

For an arbitrary polyhedral cone $\lambda\subset N_{\QQ}\oplus \QQ$, we denote by $\lambda(1)$ the set of primitive generators in $N\oplus\ZZ$ of the rays of $\lambda$. The \emph{fundamental parallelotope} of $\lambda$ is the set
\[
	P_{\lambda} := \left\{\begin{array}{c|c} \sum_{\rho\in \lambda(1)}\mu_{\rho}\rho\in N\oplus\ZZ &  0\leq \mu_{\rho}<1, \, \rho\in\lambda(1) \end{array}\right\}\subset N_{\QQ}.
\]
We also denote by the symbol $\prop(\lambda)$ the set of proper faces of $\lambda$. 
For an element $v\in N\oplus \ZZ$ we write $\chi^{v}$ for the
Laurent monomial associated with $v$. We define two functions $L_{1}$ and $L_{2}$ via the equalities
$$L_{1}(\gamma) := \left( \sum_{v\in P_{\gamma}}\chi^{v}\right)\prod_{\rho\in\tau(1)\setminus \gamma(1)}(1-\chi^{\rho})
\:\:\text{and}\:\: 
	L_{2}(\gamma) := L_{1}(\gamma) - \sum_{\gamma'\in\prop(\gamma)}L_{1}(\gamma'),
$$
for every cone $\gamma\in \Sigma_{\tau}$. We also introduce
$
	Q(\tau, \Sigma_{\tau}): = \sum_{\gamma\in \Sigma_{\tau}\setminus \prop(\tau)}L_{2}(\gamma).
$
\begin{theorem}\label{t:integral-points}
		Let $\tau\subset N_{\QQ}\oplus\QQ$ be a strongly convex simplicial polyhedral cone. 
		Then 
		\[
			\sum_{u \in \tau \cap N\oplus\ZZ} \chi^u = \frac{\sum_{n \in P_\tau} \chi^n}{\prod_{\rho\in\tau(1)}(1-\chi^{\rho})}.
		\]
\end{theorem}
\begin{proof} 
Follows from Gordan Lemma's proof, see \cite[Proposition 1.2.17]{Cox:ToricVarieties} and \cite[Section 2]{Bri88}.
\end{proof}

We give an interpretation of the polynomial $Q(\tau, \Sigma_{\tau})$ in terms of the generating function
\[
	S(\tau) = \sum_{u\in \tau^{\circ}\cap (N\oplus\ZZ)}\chi^{u},
\]
where $\tau^{\circ}$ is the relative interior of $\tau$.

\begin{lemma}\label{l:generating-function-rat-cone}
With the same notation as above,
\[
	S(\tau) = Q(\tau, \Sigma_{\tau})\prod_{\rho\in\tau(1)}(1-\chi^{\rho})^{-1}.
\]
The polynomial $Q(\tau, \Sigma_{\tau})$ does not depend on the choice of $\Sigma_{\tau}$ and it will be denoted by $Q(\tau)$.

\end{lemma}

\begin{lemma}\label{l:vertex}
Let $\D$ be a proper $\sigma$-polyhedral divisor on $\PP^{1}$. If $\rho$ is a ray of $\sigma$ such
that $\deg(\D)\cap\rho\neq \emptyset$, then there exists $\lambda\in\QQ_{>0}$ and a vertex $v\in \deg(\D)$ such that $\rho = \lambda v$. 
\end{lemma}
\begin{proof}
Straightforward.
\end{proof}

Using the combinatorial description of the log terminal condition given in Theorem~\ref{t:canonical}, we study the sign of $\omega_{X}$ on each cone of the fan $\DF_{y}$, where $y$ is a special point of $\DF$. 

\begin{lemma}\label{l:neg-omega}
Let $\tau\in \DF_{y}$ and $\rho\in \tau(1)$. Then $\omega_{X}(y, \rho)<0$.
\end{lemma}

\begin{proof}
Consider $(\D, \F) \in \DF$ such that $\rho$ is
a primitive generator of a ray of $C_{y}(\D)$. Assume first that $\rho\in \sigma(1)$, where $\sigma$ is the tail of $\D$.  If $\rho\in \Ray(\DF)$ or if a color of $\F$ is mapped onto the ray $\QQ_{\geq 0}\rho$, then by Proposition~\ref{pd:support} the rational number $\omega_{X}(y, \rho)$ is negative. 
Otherwise, by Theorem~\ref{t:canonical} the log terminal condition implies
that the locus of $\D$ is the projective line $\PP^{1}$ and that $\deg(\D)\cap \QQ_{\geq 0}\rho\neq \emptyset$. By Lemma~\ref{l:vertex}, there exists
a vertex $v$ of $\deg(\D)$ 
such that $\rho = \mu v$ for some $\mu\in\QQ_{>0}$. 

Let us write $\D = \sum_{z\in \PP^{1}}\Delta_{z}\cdot [z]$. Then (see Definition~\ref{d:support-fun}), we can find $e\in M$, $\alpha\in \QQ_{>0}$ and $f\in\CC(\PP^{1})^{*}$ such that for any $z\in\PP^{1}$ and any vertex $v$ in $\Delta_{z}$ we have 
\begin{equation}\label{e:kappa}
\theta_{X}(z, \kappa(v)v, \kappa(v)) = \alpha \kappa(v)(\langle e, v\rangle  + \ord_{z}(f)) = \kappa(v)b_{z} + \kappa(v) - 1,
\end{equation}
where $K_{\PP^{1}} = \sum_{z\in \PP^{1}}b_{z}\cdot [z]$
is a canonical divisor. Let $(v_{z})_{z\in\PP^{1}}$ be a sequence of elements of $N_{\QQ}$
such that for any $z\in\PP^{1}$, $v_{z}$ is a vertex of $\Delta_{z}$ and 
$v=\sum_{z\in\PP^{1}}v_{z}$. By Equation~\eqref{e:kappa} and Proposition~\ref{pd:support} we obtain  
\begin{equation}\label{e:omega}
	\omega_{X}(C,v,0)
 = \sum_{z\in\PP^{1}}\frac{1}{\alpha \kappa(v_{z})}\left( \kappa(v_{z})(\langle e, v_{z}\rangle  + \ord_{z}(f)\right) \\
= \frac{1}{\alpha}\left(\deg\,K_{\PP^{1}} + \sum_{z\in\PP^{1}}\left(1-\frac{1}{\kappa(v_{z})}\right)\right)<0.
\end{equation}
The inequality is a consequence of the fact that $\deg\, K_{\PP^{1}} = -2$
and of the log terminal assumption on $X$ (see Theorem~\ref{t:canonical}). If $\rho$ does not belong to $N_{\QQ}$,
then $\rho = (\kappa(w)w,\kappa(w))$ for some vertex $w$ of $\Delta_{y}$, and finally
$\omega_{X}(y, \rho) = \theta_X(y,\kappa(w)w,\kappa(w)) = -1$.
\end{proof}
\begin{corollary}\label{c:formal-completion}
Let $y\in \Sp(\DF)$ be a special point and let $\tau\in \DF$. Let $\CC[[\tau\cap (N\oplus\ZZ)]]$ be the formal completion of the ring 
$\CC[\tau\cap (N\oplus\ZZ)]$ with respect to the ideal $I_{\tau}$ generated
by the set $\{\chi^{v}\,|\,v\in\tau\cap (N\oplus\ZZ)\setminus\{0\}\}$. Elements of $I_\tau$ are of the form
$\sum_{v\in \tau\cap (N\oplus\ZZ)}a_{v}\,\chi^{v}$ with $a_{v}\in\CC$.
Let $m := \min\{ a\in\ZZ_{>0} \mid a\cdot  \omega_{X} \text{ takes its values in }\ZZ\}$.
Then the map
\[
	\varphi_\tau: (\CC[\tau\cap (N\oplus\ZZ)], I_{\tau})\rightarrow (\CC[\LL^{-\frac{1}{m}}], \LL^{-\frac{1}{m}}\,\CC[\LL^{-\frac{1}{m}}]), \; \chi^{v} \mapsto \LL^{\omega_{X}(y,v)}
\]
is a well-defined continuous morphism, and therefore it extends to the formal completions
\[
	\bar{\varphi_\tau}: \CC[[\tau\cap (N\oplus\ZZ)]] \rightarrow \CC[[\LL^{-\frac{1}{m}}]], \; \chi^{v} \mapsto \LL^{\omega_{X}(y,v)}.
\]
\end{corollary}

\begin{proof}
By Lemma~\ref{l:neg-omega} we have $\varphi_{\tau}(I_{\tau})\subset \LL^{-\frac{1}{m}}\,\CC[\LL^{-\frac{1}{m}}]$ and $\bar{\varphi_\tau}$ exists (see \cite[\S (23.H)]{Matsumura:CommAlg}).
\end{proof}

Inspired by the theory of Stanley--Reisner rings (see also \cite[Section 6]{Batyrev-Moreau}), we introduce 

\begin{definition}
Let $y\in \Sp(\DF)$ and consider $\tau\in\DF_{y}$. 
Let $m\in \ZZ_{>0}$ as in \ref{c:formal-completion}.
According to \emph{loc. cit.}
we have $\bar{\varphi}_{\tau}(Q(\tau))\in \ZZ[t^{-1}]$, where $t = \LL^{\frac{1}{m}}$ and $Q(\tau) = Q(\tau, \Sigma_{\tau})$ is the function of Lemma~\ref{l:generating-function-rat-cone}. The \emph{Stanley--Reisner polynomial} associated with the pair $(\tau, \omega_{X})$ is the polynomial 
\[
	P(\tau,  \omega_{X}):= P(\tau,  \omega_{X})(t) = \LL^{\eta(\tau, \omega_{X})}\cdot \bar{\varphi}_{\tau}(Q(\tau)) \in \ZZ[t],
\]
where $\eta(\tau, \omega_{X})\in \frac{1}{m}\ZZ$ is the degree of $P(\tau,  \omega_{X})$ divided by $m$. 
\end{definition}

The following result discribes a rational form of the stringy motivic volume $\E(X)$ in terms of the function $\omega_{X}$. 
\begin{theorem}\label{t:candidate-poles}
For every $y\in \Sp(\DF)$, let us denote by $\DF_y^{*}$ the set of cones $\tau$ of $\DF_{y}$ such that $\tau\not\subset N_{\QQ}$. Let $\Gamma = C\setminus \Sp(\DF)$ and consider the tail fan $\Sigma(\DF)$ of $\DF$. Then the stringy motivic volume of $X = X(\DF)$ is 
\begin{align*}
\E(X) =& [G/H][\Gamma]\sum_{\tau\in \Sigma(\DF)}P(\tau,  \omega_{X})\LL^{-\eta(\tau, \omega_{X})}\prod_{\rho\in\tau(1)}\left(1-\LL^{\omega_{X}(C,\rho, 0)}\right)^{-1}\\
&+ [G/H](\LL-1)\sum_{y\in\Sp(\DF)}\sum_{\tau\in \DF^{*}_{y}}P(\tau,  \omega_{X})\LL^{-\eta(\tau, \omega_{X})}\prod_{\rho\in\tau(1)}\left(1-\LL^{\omega_{X}(y,\rho)}\right)^{-1},
\end{align*}
where  $P(\tau,  \omega_{X})$ is the Stanley--Reisner polynomial associated with the pair $(\tau, \omega_{X})$.
\end{theorem}

\begin{proof}
This follows from Theorem~\ref{t:stringy-volume} and Corollary~\ref{c:formal-completion}. 
\end{proof}

As a consequence of the theorem, under some assumptions on $\DF$, we obtain also a formula for the stringy Euler characteristic of the variety $X$.

\begin{corollary}
Let $(M,I)$ be the pair describing the horospherical homogeneous space $G/H$. Consider $W := N_{G}(Q)/Q$ the Weyl group of $(G,Q)$ and $W_{I}\subset W$ the subset defined in Section~\ref{s:horo-def}. Let $r := \dim(N_{\QQ})$.  Assume that for every tail cone $\tau\in \Sigma(\DF)$ and every $\tau'\in \DF_{y}$ for $y\in\Sp(\DF)$, we have $r\geq |\tau(1)|$ and  $r+1\geq |\tau'(1)|$.
Then
\begin{align*}
	\frac{|W_{I}|}{|W|} \, e_{st}(X) = & e(\Gamma)\sum_{\tau\in \Sigma(\DF), \dim(\tau) = r}P(\tau,  \omega_{X})(1)\prod_{\rho\in\tau(1)}\frac{1}{(-\omega_{X}(C,\rho,0))}\\
 &+ \sum_{y\in\Sp(\DF)}\:\:\sum_{\tau\in \DF^{*}_{y}, \dim(\tau) = r+1}P(\tau,  \omega_{X})(1)\prod_{\rho\in\tau(1)}\frac{1}{(-\omega_{X}(y,\rho))},
\end{align*}
where  $P(\tau,  \omega_{X})$ is the Stanley--Reisner polynomial associated with the pair $(\tau, \omega_{X})$.
\end{corollary}

\section{Stringy Euler characteristic and a smoothness criterion}

In this section we start in~\ref{s:hypersurface} by illustrating Theorem~\ref{t:stringy-volume} on the example of a hypersurface endowed with a $(\CC^{*})^{2}$-action which was initially studied by Liendo and S\"u\ss~\cite[Example 1.1]{LiendoSuss}. Then, analogously to \cite[Theorem 5.3]{Batyrev-Moreau}, in~\ref{s:smoothness} we deduce from Theorem~\ref{t:stringy-volume} a smoothness criterion for locally factorial horospherical varieties of complexity one.

\subsection{An example where the acting group is a torus}\label{s:hypersurface}

	Let $N=\ZZ^2$ and $\sigma=\Cone((1,0),(1,6)) \subset N_\QQ = \QQ^2$. Define a $\sigma$-polyhedral divisor on $(\PP^1,(\CC^*)^2)$ by $\D = \sum_{y \in \PP^1} \Delta_y \cdot [y]$, where
	\begin{align*}
		\Delta_y = \begin{cases}
						\Conv((1,0),(1,1)) +  \sigma & \text{if $y=0$,} \\
						(-\frac{1}{2},0) + \sigma & \text{if $y=1$,} \\
						(-\frac{1}{3},0) + \sigma & \text{if $y=\infty$,} \\
						\sigma & \text{otherwise.}
					\end{cases}
	\end{align*}
	
	
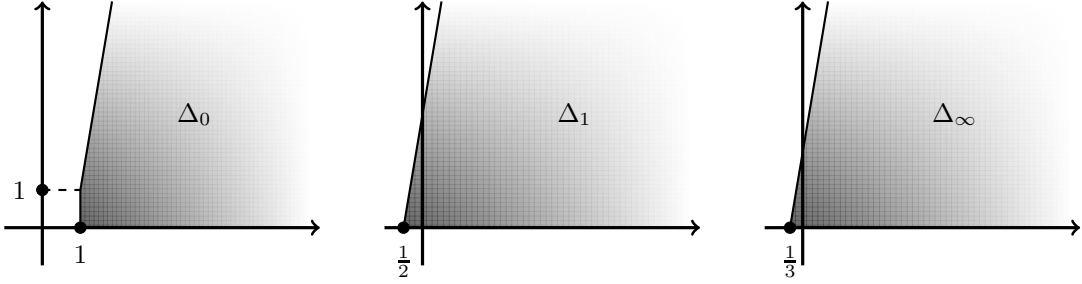
\begin{figure}
		\centering
		\begin{tikzpicture}[scale=1/2]
			\shade[lower right=white, upper left=white, lower left=black, upper right=white,fill opacity=0.6] (1,0) -- (7,0) -- (7,6) -- (11/6,6) -- (1,1) -- (1,0) ;
			\shade[lower right=white, upper left=white, lower left=black, upper right=white,fill opacity=0.6] (9.5,0) -- (17,0) -- (17,6) -- (10.5,6) -- (9.5,0);
			\shade[lower right=white, upper left=white, lower left=black, upper right=white,fill opacity=0.6] (59/3,0) -- (27,0) -- (27,6) -- (62/3,6) -- (59/3,0);
			\draw[very thick,->]	(-1,0) -- (7.3,0);
			\draw[very thick,->]	(0,-1) -- (0,6);
			\draw[thick]	(1,0) -- (7,0)
						(1,0) -- (1,1)
						(1,1) -- (11/6,6);
			\draw[very thick,->]	(9,0) -- (17.3,0);
			\draw[very thick,->]	(10,-1) -- (10,6);
			\draw[thick]	(9.5,0) -- (17,0)
						(9.5,0) -- (10.5,6);
			\draw[very thick,->]	(19,0) -- (27.3,0);
			\draw[very thick,->]	(20,-1) -- (20,6);
			\draw[thick]	(59/3,0) -- (27,0)
						(59/3,0) -- (62/3,6);
			\draw (4,3) node {$\Delta_0$};
			\draw (14,3) node {$\Delta_1$};
			\draw (24,3) node {$\Delta_\infty$};
			\node[draw,fill,circle,inner sep=1.5pt] at (1,0) {};
			\node[draw,fill,circle,inner sep=1.5pt] at (0,1) {};
			\node[draw,fill,circle,inner sep=1.5pt] at (9.5,0) {};
			\node[draw,fill,circle,inner sep=1.5pt] at (59/3,0) {};
			\draw[thick,dashed] (0,1) -- (1,1);
			\draw (1,-0.7) node {$1$};
			\draw (-0.6,1) node {$1$};
			\draw (9.5,-0.9) node {$\frac{1}{2}$};
			\draw (59/3,-0.9) node {$\frac{1}{3}$};
		\end{tikzpicture}
	\end{figure}	
	
	The variety $X(\D)$ is a hypersurface in $\AA^4$, given by the equation $x_2^3-x_3^2+x_1 x_4=0$, where the $(\CC^*)^2$-action is given by
	\[
		(\lambda_1,\lambda_2) \cdot (x_1,x_2,x_3,x_4) = (\lambda_2 x_1,\lambda_1^2 x_2,\lambda_1^3 x_3,\lambda_1^6 \lambda_2^{-1} x_4).
	\]
	An easy computation shows that the stringy motivic volume of $X(\D)$ is
	\begin{align*}
		\E(X) 
		&= \frac{\LL(\LL-1)^2}{(1-\LL^{-5})^2}(1-2\LL^{-1}+3\LL^{-2}+3\LL^{-3}+4\LL^{-4}+10\LL^{-5}-10\LL^{-6}+15\LL^{-7}+3\LL^{-8}+2\LL^{-9}+\LL^{-10})
.\end{align*}
	Thus $e_{st}(X) = 6/5> e(X) = 1$.	
\subsection{Stringy Euler characteristic and a smoothness criterion}\label{s:smoothness}
\begin{lemma}\label{l:stringy-euler}
	Let $(\D,\F)$ be a colored $\sigma$-polyhedral divisor on $(C,G/H)$ with affine locus $C_0$. Assume that $X(\D)$ is locally factorial (that is, that any Weil divisor is Cartier), and that $\sigma$ has dimension $\dim X-1$. Then the stringy Euler characteristic is equal to
	\[
		e_{st}(X) = e(C_0) \frac{|W|}{|W_I|\prod_{D_\alpha \in \F} a_\alpha},
	\]
	where, as usual, $W$ denotes the Weyl group of $(G,Q)$ and $W_I$ denotes the Weyl group of $(N_G(H),Q)$.
\end{lemma}

\begin{proof}
	Let $\Gamma \subset C_0\setminus \Sp(\D)$ be an open subset. By Theorem~\ref{t:stringy-volume} we have
	\begin{align*}
		\E(X) &= [\Gamma] [G/H] \sum_{\nu \in \sigma \cap N} \LL^{\omega_X(C,\nu,0)} + (\LL-1) [G/H] \sum_{y \in C_0 \setminus \Gamma} \: \sum_{(\nu,\ell) \in C_y(\D) \setminus (\sigma \times \{0\})}\LL^{\omega_X(y,\nu,\ell)}\\
		&= [\Gamma] \E(X(\sigma,\F)) + \sum_{y \in C_0 \setminus \Gamma} \left( \E(X(C_y(\D),\F)) -(\LL-1)\E(X(\sigma,\F)) \right),
	\end{align*}
	where $X(\sigma,\F)$ (respectively $X(C_y(\D),\F)$) is the $G$-equivariant (respectively $G\times \CC^*$-equivariant) embedding of the horospherical homogeneous space $G/H$ (respectively $G/H \times \CC^*$) associated with the colored cone $(\sigma,\F)$ (respectively $(C_y(\D),\F)$). Here the last equality comes from \cite[Theorem 4.3]{Batyrev-Moreau}.
	
	Passing to the stringy $E$-polynomial and evaluating at $u=v=1$, we obtain
	\[
		e_{st}(X) = e(\Gamma) e_{st}(X(\sigma,\F)) + \sum_{y \in C_0 \setminus \Gamma} e_{st}(X(C_y(\D),\F)).
	\]
	Using \cite[Proposition 5.11]{Batyrev-Moreau} we see that $e_{st}(X(C_y(\D)),\F)=e_{st}(X(\sigma,\F))$ since the stringy Euler characteristic only depends on the Weyl groups and the colors, which are unchanged.
	We deduce
	\[
		e_{st}(X) = e(C_0) e_{st}(X(\sigma,\F))\:\:\text{and}\:\: e_{st}(X(\sigma,\F)) = \frac{|W|}{|W_I| \prod_{D_\alpha \in \F} a_\alpha }
	\]
		which concludes the proof.
\end{proof}

The next proposition gives a full description of the $G$-orbits of a simple $G$-model of $C \times G/H$ corresponding to a colored $\sigma$-polyhedral divisor $(\D,\F)$ with affine locus $C_0$. For the complexity zero case we refer to \cite[Proposition 2.4]{Batyrev-Moreau}.

\begin{proposition}\label{p:G-orbits}
	Let $(\D,\F)$ be a colored $\sigma$-polyhedral divisor on $(C,G/H)$ with affine locus $C_0$. Then $X:=X(\D,\F)$ has two types of $G$-orbits :
	\begin{enumerate}[(i)]
		\item \emph{horizontal orbits}, which are contained in the $G$-invariant open subset $\Gamma \times X(\sigma,\F)$ of $X$, where $\Gamma = C_0 \setminus \Sp(\D)$ and $X(\sigma,\F)$ is the $G/H$-embedding associated with the colored cone $(\sigma,\F)$,
		\item \emph{vertical orbits}, which are the remaining $G$-orbits.
	\end{enumerate}
	
	Moreover, the horizontal $G$-orbits of $X$ are parametrized by triples $(y,\tau,\F_\tau)$, where $y \in \Gamma$, $\tau$ is a face of $\sigma$, and $\F_\tau = \{ D \in \F \mid \varrho(D) \in \tau \}$. The stabilizer $H_{y,\tau,\F_\tau}=:H_\tau$ of such a triple is given by
	\[
		H_\tau = \{ g \in P_{I_{\F_\tau}} \mid \chi^m(g)=1 \text{ for any } m\in \tau^\perp \cap M \},
	\]
	where $I_{\F_\tau} \subset \Phi$ is the reunion with $I$ of the set of simple roots indexing the colors in $\F_\tau$, and $\chi^m$ is the character of $P_{I_{\F_\tau}}$ associated with $m$.
	Finally, the vertical $G$-orbits of $X$ are parametrized by triples $(y,F,\F_F)$, where $y \in \Sp(\D)$, $F$ is a face of the $\sigma$-polyhedron $\Delta_y$ of $\D$ associated with $y$, and
	\[
		\F_F = \{ D_\alpha \in \F \mid \varrho(D_\alpha) \in \lambda(F) \} \:\: \text{with} \:\:
		\lambda(F) = \{ m \in M_\QQ \mid \langle m,v - v' \rangle \geq 0 \text{ for any } v \in \Delta_y, v' \in F \}.
	\]
	The stabilizer $H_{y,F,\F_F}=:H_{y,F}$ of a triple $(y,F,\F_F)$ is given by
	\[
		H_{y,F} = \{ g \in P_{I_{\F_F}} \mid \forall m\in M_y, \chi^m(g)=1 \}
		\:\:\text{with}\:\:
		M_y = \left\{ m \in M \cap \lambda(F) \cap (-\lambda(F)) \,\middle|\, \min_{v \in \Delta_y} \langle m,v\rangle \in \ZZ \right\}.
	\]
\end{proposition}

\begin{proof} This follows by combining \cite[Proposition 2.4]{Batyrev-Moreau} and \cite[Section 7]{AltmannHausen}.
\end{proof}

We compute the (usual) Euler characteristic and compare it to the stringy version of Lemma~\ref{l:stringy-euler}.

\begin{lemma}\label{l:euler}
	Under the same assumptions as in Lemma~\ref{l:stringy-euler}, the Euler characteristic of $X$ is 
	\[
		e(X) = e(C_0)\frac{|W|}{|W_{I_\F}|}.
	\]
\end{lemma}

\begin{proof}
	By Proposition~\ref{p:G-orbits} we have a decomposition $e(X) = e(\Gamma \times X(\sigma,\F)) + \sum_{O \text{ vertical $G$-orbit}} e(O).$ Using \cite[Proposition 5.11]{Batyrev-Moreau} and the parametrization of vertical orbits in Proposition~\ref{p:G-orbits}, we obtain
	\[
		e(X) = e(\Gamma) \frac{|W|}{|W_{I_\F}|} + \sum_{y \in \Sp(\D)} \sum_{v \in \Delta_y} \frac{|W|}{\left|W_{I_{\F_{\{v\}}}}\right|}.
	\]
	Thus it only remains to prove that 
	\begin{equation}\label{e:euler-identity}
		\frac{|W|}{|W_{I_\F}|} = \sum_{v \in \Delta_y} \frac{|W|}{\left|W_{I_{\F_{\{v\}}}}\right|}.
	\end{equation}
	Fix a point $y_0 \in C_0$ and define  $\hat \D := \sum_{z \in \AA^1} \hat \Delta_z \cdot [z]$ a colored $\sigma$-polyhedral divisor $(\hat \D,\F)$ on $(\PP^1,G/H)$ where $\hat \Delta_z = \Delta_{y_0}$ if $z=0$ and otherwise  $\hat \Delta_z = \sigma$.
		The variety $X(\hat \D)$ is locally factorial, horospherical, and it identifies with the $(G \times \CC^*)$-equivariant embedding of $G/H \times \CC^*$ associated with the colored cone $(C_0(\hat \D),\F)=(C_y(\D),\F)$. Thus by \cite[Proposition 5.11]{Batyrev-Moreau}, we obtain
	\[
		\frac{|W|}{|W_{I_\F}|} = e(X(C_0(\hat \D),\F)) = e(X(\hat \D)) = e(\CC^*) e(X(\sigma,\F)) + \sum_{v \in \Delta_y} \frac{|W|}{\left|W_{I_{\F_{\{v\}}}}\right|}.
	\]
	Since $e(\CC^*)=0$, this proves Equation~\eqref{e:euler-identity} and concludes the proof.
\end{proof}

Combining Lemmas~\ref{l:stringy-euler} and~\ref{l:euler} with the smoothness criterion of \cite[Theorem 2.5]{LangloisTerpereau}, we obtain 

\begin{theorem}\label{t:smoothness}
	Let $(\D,\F)$ be a colored $\sigma$-polyhedral divisor on $(C,G/H)$ with affine locus $C_0 \subset C$. Assume that $X=X(\D)$ is locally factorial and that for any $y \in C_0$, the Cayley cone $C_y(\D)$ has dimension $d=\dim X$. Then we have $e_{st}(X) \geq e(X)$. Moreover, if $2-2g-|C \setminus C_0| \neq 0$, then $X$ is smooth if and only if $e_{st}(X)= e(X)$.
\end{theorem}

\section*{Acknowledgements}

The authors would like to thank Anne-Sophie Kaloghiros, Hendrik S\"u{\ss},
Willem Veys and Takehiko Yasuda for useful discussions. 

This research is supported by the Max Planck Institute for Mathematics of Bonn, the LABEX MILYON (ANR-10-LABX-0070), the project ANR-15-CE40-0008 (D\'efig\'eo) and the ERC Grant Agreement nr. 246903 NMNAG.

\bibliographystyle{amsalpha}
\bibliography{bibname}
\end{document}